\documentclass[12pt]{article} 
 \usepackage[margin=1in]{geometry}
\usepackage{setspace}
\usepackage{graphicx}	
\usepackage{subcaption}
\usepackage{float} 
\usepackage[toc]{appendix}
\usepackage[draft]{todonotes} 
\usepackage{enumerate}	
\usepackage{paralist}	
\usepackage{amsmath, amsthm, amssymb,mathtools}
\usepackage{thm-restate}

\usepackage[pdftex,hypertexnames=false,linktocpage=true,hidelinks]{hyperref}
\usepackage{cleveref}

\newtheorem{thm}{Theorem}[section]

\newtheorem{prop}[thm]{Proposition}
\newtheorem{claim}{Claim}
\newtheorem{lemma}[thm]{Lemma}

\newtheorem{prob}[thm]{Problem}
\theoremstyle{definition}

\theoremstyle{definition}
\newtheorem{const}{Construction}

\usepackage{tikz}
\usepackage{xcolor}

\usetikzlibrary{decorations.pathreplacing}

\definecolor{amber}{rgb}{1.0, 0.75, 0.0}

\definecolor{forest}{rgb}{0.0, 0.5, 0.0}

\definecolor{cadmium}{rgb}{0.93, 0.53, 0.18}

\definecolor{byzantine}{rgb}{0.74, 0.2, 0.64}

\definecolor{brilliantrose}{rgb}{1.0, 0.33, 0.64}

\definecolor{caribbeangreen}{rgb}{0.0, 0.8, 0.6}

\definecolor{electriccyan}{rgb}{0.0, 1.0, 1.0}

\definecolor{periwinkle}{rgb}{0.8, 0.8, 1.0}

\definecolor{steelblue}{rgb}{0.27, 0.51, 0.71}

\title{Proper Rainbow Saturation Numbers for Cycles}

\begin{document}

\author{Anastasia Halfpap\thanks{
Department of Mathematics, Iowa State University, Ames, IA, USA.  E-mail \texttt{ahalfpap@iastate.edu}
}
\and 
Bernard Lidick\'y\thanks{
Department of Mathematics, Iowa State University, Ames, IA, USA.  E-mail \texttt{lidicky@iastate.edu}
}
\and
Tom\'a\v{s} Masa\v{r}\'ik\thanks{Institute of Informatics, Faculty of Mathematics, Informatics and Mechanics, University of Warsaw, Warszawa, Poland. E-mail \texttt{masarik@mimuw.edu.pl}}
}

\maketitle

\begin{abstract}
We say that an edge-coloring of a graph $G$ is \textit{proper} if every pair of incident edges receive distinct colors, and is \textit{rainbow} if no two edges of $G$ receive the same color. Furthermore, given a fixed graph $F$, we say that $G$ is \textit{rainbow F-saturated} if $G$ admits a proper edge-coloring which does not contain any rainbow subgraph isomorphic to $F$, but the addition of any edge to $G$ makes such an edge-coloring impossible. The maximum number of edges in a rainbow $F$-saturated graph is the rainbow Tur\'an number, whose study was initiated in 2007 by Keevash, Mubayi, Sudakov, and Verstra\"ete. Recently, Bushaw, Johnston, and Rombach introduced study of a corresponding saturation problem, asking for the \textit{minimum} number of edges in a rainbow $F$-saturated graph. We term this minimum the \textit{proper rainbow saturation number} of $F$, denoted $\mathrm{sat}^*(n,F)$. 
We asymptotically determine $\mathrm{sat}^*(n,C_4)$, answering a question of Bushaw, Johnston, and Rombach. We also exhibit constructions which establish upper bounds for $\mathrm{sat}^*(n,C_5)$ and $\mathrm{sat}^*(n,C_6)$.

\end{abstract}

\section{Introduction}

A central problem in extremal graph theory is to understand the set of $n$-vertex graphs $G$ which do not contain some forbidden subgraph $F$. Formally, given graphs $G, F$, we say that $G$ \textit{contains} a copy of $F$ (or \emph{$F$-copy}) if $G$ contains a subgraph (not necessarily induced) isomorphic to $F$; if $G$ does not contain a copy of $F$, we say that $G$ is \textit{F-free}. Note that if $G$ is $F$-free, then all subgraphs of $G$ are also $F$-free; thus, it is natural to restrict our attention to edge-maximal $F$-free $n$-vertex graphs, as these contain all $F$-free $n$-vertex graphs. We use $V(G), E(G)$ to denote the vertex and edge sets of $G$, and for $x,y \in V(G)$ such that $xy \not\in E(G)$, we denote by $G + xy$ the graph on vertex set $V(G)$ with edge set $E(G) \cup \{xy\}$. A graph $G$ is \textit{F-saturated} if $G$ is $F$-free but, for any $e \not\in E(G)$, the graph $G + e$ contains a copy of $F$. 

Further restricting our focus, we may ask which $F$-saturated $n$-vertex graphs are somehow extremal. That is, in the set of $n$-vertex $F$-saturated graphs, which elements optimize some graph parameter? This question yields two natural avenues of research. The \textit{Tur\'an number} $\mathrm{ex}(n,F)$ is the maximum number of edges among all $n$-vertex, $F$-saturated graphs. Famously first considered by Mantel~\cite{Mantel} in the case $F = K_3$, and for general cliques by Tur\'an~\cite{Turan41}, the study of $\mathrm{ex}(n,F)$ remains a vibrant area of study in its own right, as well as giving rise to a variety of natural generalizations and variations. On the other hand, the \textit{saturation number} $\mathrm{sat}(n,F)$ is the \textit{minimum} number of edges among all $n$-vertex, $F$-saturated graphs. The study of saturation numbers is also well-established (see, for instance, \cite{satsurvey}) and, like Tur\'an problems, it is natural to generalize saturation problems to a variety of contexts. 

This paper concerns saturation problems in an edge-colored setting. An \textit{edge-coloring} of a graph $G$ is a function $c: E(G) \rightarrow \mathbb{N}$. We say that $c(e)$ is the \textit{color} of $e$, and that $c$ is a \textit{proper} edge-coloring if for any two incident edges $e,f$, we have $c(e) \neq c(f)$. An edge-colored graph is \textit{rainbow} if all of its edges receive different colors. Given fixed graphs $G$, $F$, and a proper edge-coloring $c$ of $G$, we say that $G$ is \textit{rainbow-F-free under} $c$ if $G$ does not contain any copy of $F$ which is rainbow with respect to $c$. Moreover, we say that $G$ is \textit{(properly) rainbow $F$-saturated} if $G$ satisfies the following conditions.

\begin{enumerate}
    \item There exists a proper edge-coloring $c$ of $G$ such that $G$ is rainbow-$F$-free under $c$;

    \item For any edge $e \not\in E(G)$, any proper edge-coloring of $G+e$ contains a rainbow $F$-copy.
\end{enumerate}

Motivated by a problem in additive number theory, Keevash, Mubayi, Sudakov, and Verstra\"ete~\cite{KMSV} introduced the \textit{rainbow Tur\'an number} $\mathrm{ex^*}(n,F)$ in 2007, which is the maximum number of edges in an $n$-vertex, rainbow $F$-saturated graph. Following the analogy between $\mathrm{ex}(n,F)$ and $\mathrm{ex}^*(n,F)$, it is natural to also consider the rainbow counterpart to $\mathrm{sat}(n,F)$. Bushaw, Johnston, and Rombach~\cite{Bushaw2022} recently initiated a study of this rainbow version of the saturation number, denoted $\mathrm{sat}^*(n,F)$, the minimum number of edges in an $n$-vertex rainbow $F$-saturated graph. We call $\mathrm{sat}^*(n,F)$ the \textit{proper rainbow saturation number} of $F$, since all edge-colorings in this setting are proper.  While slightly lengthy, this terminology distinguishes $\mathrm{sat}^*(n,F)$ from an already-studied function which has been termed the rainbow saturation number in the literature (see, e.g.,~\cite{barruscoloredsat}, \cite{behague2022rainbow}, \cite{giraorainbowsat}), and which does not assume a setting of proper edge-colorings. 

Given that consideration of $\mathrm{sat}^*(n,F)$ is extremely new, few results have been established in the area, and the general behavior of $\mathrm{sat}^*(n,F)$ remains unclear. For instance, while it is simple to observe that $\mathrm{ex}(n,F) \leq \mathrm{ex}^*(n,F)$ for all $F$, it is not obvious whether we have $\mathrm{sat}(n,F) \leq \mathrm{sat}^*(n,F)$ for all $F$. The following theorem illustrates that we sometimes have $\mathrm{sat}(n,F) < \mathrm{sat}^*(n,F)$, and in fact $\mathrm{sat}^*(n,F)$ may differ from $\mathrm{sat}(n,F)$ by a multiplicative factor. Here and throughout, we denote by $P_{\ell}$ the path on $\ell$ \textit{vertices} (that is, on $\ell - 1$ edges). 

\begin{thm}[{\cite[Theorem 3.5]{Bushaw2022}}]

For each $n \geq 16$, we have 
\[\lfloor \frac{4n}{5}\rfloor \leq \mathrm{sat}^*(P_4, n) \leq \frac{4}{5}n - \frac{17}{10}c,\]
where $0 \leq c \leq 4$ and $c \equiv -n \mod 5$.

\end{thm}

For contrast, $\mathrm{sat}(n,P_4)$ is approximately $\frac{n}{2}$.

\begin{thm}[{\cite[Proposition 5]{KaszonyiTuza}}]

\[\mathrm{sat}(n,P_4) = \begin{cases} \frac{n}{2} & \text{ if $n$ is even} \\ \frac{n+3}{2} & \text{ if $n$ is odd} \end{cases}\]

\end{thm}

Apart from $\mathrm{sat}^*(P_4, n)$, we do not have tight bounds on any proper rainbow saturation number, except for trivial cases where every proper coloring of $F$ is rainbow and thus $\mathrm{sat}(n,F) = \mathrm{sat}^*(n,F)$. The goal of this paper is to contribute to the understanding of the proper rainbow saturation numbers of cycles, in particular by determining $\mathrm{sat}^*(n,C_4)$ asymptotically. The previous best known bounds on $\mathrm{sat}^*(n,C_4)$ are due to Bushaw, Johnston, and Rombach~\cite{Bushaw2022}.

\begin{thm}[{\cite[Theorem 3.6]{Bushaw2022}}]

For $n \geq 4$, we have
\[n \leq \mathrm{sat}^*(n,C_4) \leq 2n-2.\]

\end{thm}

For comparison, the ordinary saturation number of $C_4$ is known exactly.

\begin{thm}[\cite{Ollman}]

$\mathrm{sat}(n,C_4) = \lfloor \frac{3n - 5}{2} \rfloor$.

\end{thm}

We contribute the following bounds on $\mathrm{sat}^*(n,C_4)$, which asymptotically determine its value and show that $\mathrm{sat}^*(n,C_4)$ is separated from $\mathrm{sat}(n,C_4)$ by a constant multiplicative factor.

\begin{restatable}{thm}{mainupper}\label{c4 upper bound}

For $n \geq 7$, 

\[\mathrm{sat}^*(n,C_4) \leq \frac{11}{6} n + O(1).\]

\end{restatable}

\begin{restatable}{thm}{mainlower}\label{c4 lower bound}
  Let $\frac{11}{45}>\varepsilon > 0$ be given. There exists $n_0 \in \mathbb{N}$ such that, for all $n \geq n_0$, 
\[\mathrm{sat}^*(n,C_4) > \left( \frac{11}{6} - \varepsilon \right) n.\]

\end{restatable}

For cycles of length greater than $4$, little is known. Unlike other saturation numbers, it is not even known whether $\mathrm{sat}^*(n,F)$ is always linear in $n$, although Bushaw, Johnston, and Rombach~\cite{Bushaw2022} describe a class of graphs with linear proper rainbow saturation number.

\begin{thm}[{\cite[Theorem 4.2]{Bushaw2022}}]\label{general bound}
Suppose that $F$ contains no induced even cycle. Then there is a constant $c$ depending only on $F$ such that $\mathrm{sat}^*(n,F) \leq cn$.
\end{thm}

Theorem~\ref{general bound} can be used to derive some upper bound on proper rainbow saturation numbers for odd cycles, but the constant $c$ given may be very large. There are no published bounds on $\mathrm{sat}^*(n,C_{2\ell})$ for $\ell > 2$. We contribute constructions that improve this state of affairs for $C_5$ and $C_6$. For $C_5$, we obtain a single bound regardless of the parity of $n$.

\begin{restatable}{thm}{fivecycle}\label{c5 bound}
For $n \geq 9, \mathrm{sat}^*(n,C_5) \leq \lfloor \frac{5n}{2}\rfloor - 4$.
\end{restatable}

For $C_6$, our bound varies slightly with the congruence class of $n$ modulo $3$. To avoid separate cases, we allow a constant error term which absorbs this discrepancy.

\begin{restatable}{thm}{sixcycle}\label{c6 bound}
$\mathrm{sat}^*(n,C_6) \leq  \frac{7}{3} n + O(1)$.
\end{restatable}

Using the elementary observation that, for $n > 2$, a rainbow $C_{\ell}$-saturated graph contains no acyclic component (since the addition of an edge either within an acyclic component yields a component containing at most one $C_{\ell}$ copy, which can be properly colored to avoid a rainbow $C_{\ell}$-copy, while an edge between distinct components creates no new cycles at all), we have the immediate lower bound $\mathrm{sat}^*(n,C_{\ell}) \geq n$ for all $\ell$ and all $n > 2$. Thus, the bounds given in Theorems~\ref{c5 bound} and~\ref{c6 bound} seem reasonable, although we do not attempt to find matching lower bounds.

The remainder of the paper is organized as follows. In Section~\ref{c4 section}, we present a construction establishing Theorem~\ref{c4 upper bound} and prove Theorem~\ref{c4 lower bound}. In Section~\ref{c56 section}, we present constructions establishing Theorems~\ref{c5 bound} and~\ref{c6 bound}.




\paragraph{Notation.}
We denote \emph{degree} of a vertex $v$ in a graph $G$ by $d_G(v)$ and the \emph{minimum degree} of a vertex $v$ in a graph $G$ by $\delta(G)$.
If $G$ is clear from context we omit the subscript and simply write $d(v)$ for the degree of $v$ in $G$.
Given vertices $u,v$ we denote \emph{distance} by $d(u,v)$.
That is, $d(u,v)$ is the minimum number of edges on a path from $u$ to $v$.
We use $N[v]$ to denote the \emph{closed neighbourhood} of a vertex $v$ and we let $N(v)\coloneqq N[v]\setminus \{v\}$.
For $S\subseteq V(G)$, we use $N(S)$ (resp.\ $N[S]$) as a shortcut for $\bigcup_{v\in S} N(v)$ (resp.\ $\bigcup_{v\in S} N[v]$).
Given a graph $G$ and $S \subset V(G)$, we use $G[S]$ to denote the subgraph of $G$ induced on $S$, that is, the graph with vertex set $S$ and edge set $E(G[S]) = \{uv \in E(G): u,v \in S \}$.

\section{Bounds for $C_4$}\label{c4 section}

We begin by improving the upper bound on $\mathrm{sat}^*(n,C_4)$, with a construction showing that $\mathrm{sat}^*(n,C_4) \leq \frac{11 n}{6} + O(1)$. Before stating the construction, we establish a variety of facts about properly rainbow $C_4$-saturated graphs, which will be useful throughout. We begin with a proposition collecting a few elementary observations.

\begin{prop}\label{observations}
Let $G$ be a rainbow $C_4$-saturated graph. Then the following hold.

\begin{enumerate}

  \item $G$ contains at most one vertex of degree $1$;\phantomsection\label{obs:deg1}

\item For any vertices $u,v \in V(G)$, $d(u,v) \leq 3$;\phantomsection\label{obs:dist}

\item For any vertices $u,v \in V(G)$, $|N(u) \cap N(v)| < 4$.\phantomsection\label{obs:nighbr}

\end{enumerate}

\end{prop}

Next, we prove the following key lemma, which will be required to demonstrate that our construction is properly rainbow $C_4$-saturated.

\begin{lemma}\label{traps}

  Let $G$ be a graph and $v \in V(G)$. If there exists a proper edge-coloring of $G$ which is rainbow $C_4$-free, then the subgraph of $G$ induced on $N(v)$ does not contain the following subgraphs (not necessarily induced); see \Cref{fig:notInNv}:

\begin{enumerate}
\item A copy of $K_3$ with pendant edges from two vertices;
\item $C_4$;
\item A copy of $C_k$ with a pendant edge, for any $k \geq 5$;
\item The double star $D_{2,2}$, or any subdivision thereof.

\end{enumerate}

\end{lemma}

\begin{figure}[h]
\begin{center}
\begin{tikzpicture}

\filldraw (0,0) circle (0.05cm);
\filldraw (-0.8,1) circle (0.05cm);
\filldraw (0.8,1) circle (0.05cm);
\filldraw (-1.5,1.5) circle (0.05cm);
\filldraw (1.5,1.5) circle (0.05cm);
\draw (0,0) -- (-0.8, 1) -- (-1.5, 1.5);
\draw (0,0) -- (0.8, 1) -- (1.5, 1.5);
\draw (-0.8,1) -- (0.8, 1);

\filldraw (3,0) circle (0.05cm);
\filldraw (3,1.5) circle (0.05cm);
\filldraw (4.5,0) circle (0.05cm);
\filldraw (4.5,1.5) circle (0.05cm);
\draw (3,0) -- (4.5,0) -- (4.5, 1.5) -- (3,1.5) -- (3,0);

\filldraw (6.5,0) circle (0.05 cm);
\filldraw (7.5,0) circle (0.05 cm);
\filldraw (8.5,0) circle (0.05 cm);
\filldraw (6,0.75) circle (0.05 cm);
\filldraw (8,0.75) circle (0.05 cm);
\filldraw (6.5,1.5) circle (0.05 cm);
\draw (6.5,1.5) -- (6,0.75) -- (6.5, 0) -- (7.5,0) -- (8,0.75);
\draw[dashed] (6.5,1.5) to[bend left = 50] (8,0.75);
\draw (7.5,0) -- (8.5,0);

\filldraw (10,0) circle (0.05 cm);
\filldraw (10,1.5) circle (0.05 cm);
\filldraw (11, 0.75) circle (0.05 cm);
\filldraw (12, 0.75) circle (0.05 cm);
\filldraw (13,0) circle (0.05 cm);
\filldraw (13,1.5) circle (0.05 cm);
\draw (10,0) -- (11,0.75) -- (12, 0.75) -- (13,0);
\draw (10,1.5) -- (11,0.75);
\draw (13,1.5) -- (12,0.75);
\end{tikzpicture}
\caption{Subgraphs not appearing in $N(v)$}\label{fig:notInNv}
\end{center}

\end{figure}
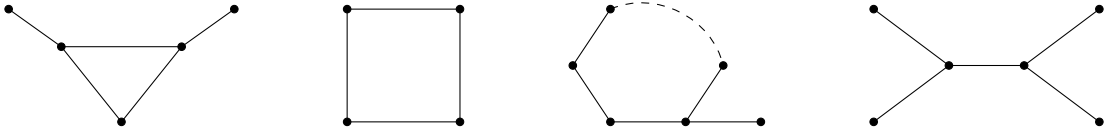

\begin{proof}

We show that if any of the above subgraphs appear in the neighborhood of $v$, then it is impossible to properly edge-color $G$ while avoiding a rainbow $C_4$-copy. 

We distinguish several cases based on $N(v)$.
\begin{enumerate}
    \item $N(v)$ contains a copy of $K_3$ with pendant edges from two vertices. 
    We label the vertices of this $K_3$ as $v_1,v_2,v_3$, and the degree $1$ vertices of the pendant edges as $v_4,v_5$. 
    We draw the configuration in Figure~\ref{forbidden1}; without loss of generality, $vv_i$ has color $i$.

    To avoid a rainbow $C_4$-copy using vertices $v,v_1,v_2,v_4$, either $c(v_1v_2) = 4$ or $c(v_2v_4) = 1$. 
    Similarly, to avoid a rainbow $C_4$-copy using vertices $v,v_3,v_2,v_4$, either $c(v_2v_3) = 4$ or $c(v_2v_4) = 3$. 
    Since $c(v_2v_4)$ cannot simultaneously be equal to $1$ and $3$, it follows that one of $c(v_1v_2),c(v_2,v_3)$ equals $4$. 
    Analogously, to avoid a rainbow $C_4$-copy using $v_5$, one of $c(v_1v_3), c(v_2,v_3)$ must equal $5$. 
    Thus, two edges of the triangle on $v_1,v_2,v_3$ form a $P_3$ colored with $4$ and $5$ which are not in $\{1,2,3\}$. 
    This cherry, along with $v$, immediately forms a rainbow $C_4$-copy.

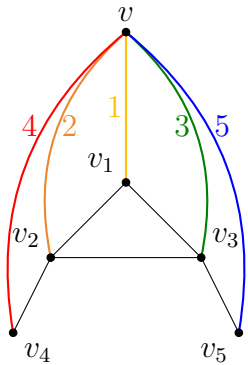
\begin{figure}
\centering
\begin{tikzpicture}

\draw[thick, color = amber] (0,0) -- (0,-2);
\draw[thick, color = amber] (-0.15,-1) node{1};
\draw[thick, color = cadmium] (0,0) to[bend right] (-1,-3);
\draw[thick, color = cadmium] (-0.75,-1.25) node{2};
\draw[thick, color = forest] (0,0) to[bend left] (1,-3);
\draw[thick, color = forest] (0.75,-1.25) node{3};
\draw[thick, color = red] (0,0) to[bend right] (-1.5,-4);
\draw[thick, color = red] (-1.28,-1.25) node{4};
\draw[thick, color = blue] (0,0) to[bend left] (1.5,-4);
\draw[thick, color = blue] (1.28,-1.25) node{5};

\filldraw (0,0) circle (0.05 cm) node[above]{$v$};

\filldraw (0,-2) circle (0.05 cm) node[above left]{$v_1$};
\filldraw (-1,-3) circle (0.05 cm) node[above left]{$v_2$};
\filldraw (1,-3) circle (0.05 cm) node[above right]{$v_3$};

\filldraw (-1.5,-4) circle (0.05 cm) node[below right]{$v_4$};
\filldraw (1.5,-4) circle (0.05 cm) node[below left]{$v_5$};

\draw (-1,-3) -- (-1.5,-4);
\draw (-1,-3) -- (1,-3) -- (1.5,-4);
\draw (1, -3) -- (0,-2);
\draw (0,-2) -- (-1,-3);

\end{tikzpicture}
\caption{$K_3$ and pendant edges in $N(v)$}\label{forbidden1}
\end{figure}

    \item $N(v)$ contains a copy of $C_4$, on vertices $v_1, v_2, v_3, v_4$. Without loss of generality, ${c(vv_i) = i}$; see Figure~\ref{forbidden2}.

    To avoid a rainbow $C_4$-copy using vertices $v,v_1,v_2,v_3$, either $c(v_1v_2) = 3$ or $c(v_2v_3) = 1$. 
    The situation is at this point symmetric, so without loss of generality, $c(v_2v_3) = 1$. 
    Now, we will have a rainbow $C_4$-copy using $v, v_2, v_3, v_4$ unless $c(v_3v_4) = 2$. 
    Similarly, we are forced to choose $c(v_1v_4) = 3$ and $c(v_1v_2) = 4$. But now $v_1,v_2,v_3,v_4$ forms a rainbow $C_4$-copy.

\begin{figure}[h]
\centering
\begin{tikzpicture}

\draw[thick, color = cadmium] (0,0) to[bend right] (-1,-3);
\draw[thick, color = cadmium] (-0.75,-1.25) node{2};
\draw[thick, color = amber] (0,0) to[bend left] (1,-3);
\draw[thick, color = amber] (0.75,-1.25) node{3};
\draw[thick, color = red] (0,0) to[bend right] (-1.5,-4);
\draw[thick, color = red] (-1.28,-1.25) node{1};
\draw[thick, color = forest] (0,0) to[bend left] (1.5,-4);
\draw[thick, color = forest] (1.28,-1.25) node{4};

\filldraw (0,0) circle (0.05 cm) node[above]{$v$};

\filldraw (-1,-3) circle (0.05 cm) node[above left]{$v_2$};
\filldraw (1,-3) circle (0.05 cm) node[above right]{$v_3$};

\filldraw (-1.5,-4) circle (0.05 cm) node[below right]{$v_1$};
\filldraw (1.5,-4) circle (0.05 cm) node[below left]{$v_4$};

\draw (-1,-3) -- (-1.5,-4);
\draw (-1,-3) -- (1,-3) -- (1.5,-4);
\draw(-1.5,-4) -- (1.5,-4);

\end{tikzpicture}
\caption{A copy of $C_4$ in $N(v)$}\label{forbidden2}
\end{figure}
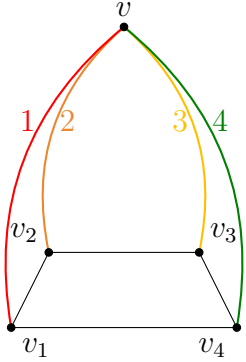

\item Let $k\geq 5$ and suppose that $N(v)$ contains a copy of $C_k$ with a pendant edge. 
We label the vertices of $C_k$ as $v_1, \dots, v_k$, and the endpoint of the pendant edge which is not contained in the cycle as $v_{k+1}$. 
Without loss of generality, $v_{k+1}$ is adjacent to $v_k$, and each edge $vv_i$ receives color $i$. 
For reference, we depict this in Figure~\ref{forbidden3}.

Observe that to avoid a rainbow $C_4$-copy using $v, v_1, v_k,$ and $v_{k+1}$, either $c(v_kv_{k+1}) = 1$ or $c(v_1v_k) = k+1$. Analogously, either $c(v_kv_{k+1}) = k-1$ or $c(v_{k-1}v_k) = k+1$. Without loss of generality, assume $c(v_kv_{k+1}) = k-1$ and $c(v_1v_k) = k+1$. Now, to avoid a rainbow $C_4$-copy using $v,v_k,v_1,v_2$, we must have $c(v_1v_2) = k$. 
Proceeding inductively, for each $i \in \{2, \dots k-1\}$, we must have $c(v_iv_{i+1}) = i-1$ to avoid a rainbow $C_4$-copy using $v,v_{i-1},v_i,v_{i+1}$. However, now $v, v_1,v_{k-1},v_k$ form a $C_4$-copy with edge colors $c(vv_1) = 1, c(v_1v_k) = k+1, c(v_kv_{k-1}) = k-2$, and $c(v_{k-1}v) = k-1$. Since $k > 3$, we must have that $1$ is distinct from $k-2, k-1, k+1$, so this $C_4$-copy is rainbow, a contradiction.

\begin{figure}[h]
\centering
\begin{tikzpicture}

\draw[thick, color = cadmium] (0,0) to[bend right] (-1,-3);
\draw[thick, color = amber] (0,0) to[bend left] (1,-3);
\draw[thick, color = red] (0,0) to[bend right] (-2,-4);
\draw[thick, color = red] (-2.08,-3) node{1};
\draw[thick, color = forest] (0,0) to[bend left] (2,-4);
\draw[thick, color = forest] (2.08,-3) node{4};
\draw[thick, color = blue] (0,0) to[bend right] (-1.5,-5);
\draw[thick, color = blue] (-1.48,-4) node{$k$};
\draw[thick, color = cadmium] (-0.92,-2.5) node{2};
\draw[thick, color = amber] (0.92,-2.5) node{3};
\draw[thick, color = byzantine] (0,0) to[bend right = 40] (-3,-7);
\draw[thick, color = byzantine] (-3.75,-4.5) node{$k+1$};

\filldraw (0,0) circle (0.05 cm) node[above]{$v$};

\filldraw (-1,-3) circle (0.05 cm) node[left]{$v_2$};
\filldraw (1,-3) circle (0.05 cm) node[right]{$v_3$};

\filldraw (-2,-4) circle (0.05 cm) node[left]{$v_1$};
\filldraw (2,-4) circle (0.05 cm) node[right]{$v_4$};

\filldraw (-1.5, -5) circle (0.05 cm) node[left]{$v_k$};

\filldraw (-3, -7) circle(0.05 cm) node[left]{$v_{k+1}$};

\draw (-3,-7) -- (-1.5,-5) -- (-2,-4) -- (-1,-3) -- (1,-3) -- (2,-4); 
\draw (-1.5,-5) -- (-0.7, -5.3);
\draw (2,-4) -- (1.7, -4.8);
\draw[dotted] (-0.7, -5.3) to[bend right] (1.7, -4.8); 


\end{tikzpicture}
\caption{A copy of $C_k$ with a pendant edge in $N(v)$}\label{forbidden3}
\end{figure}
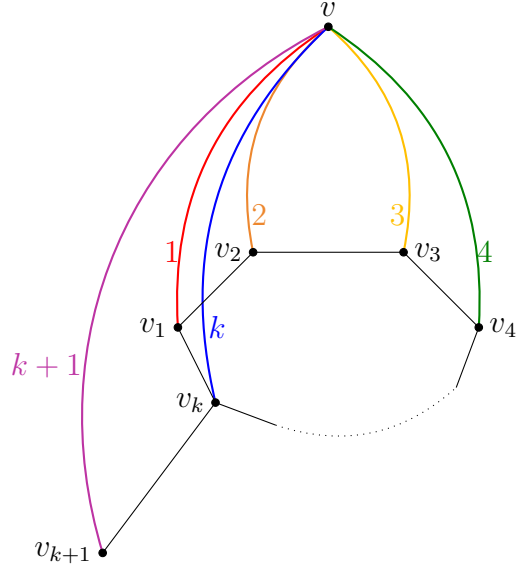

\item $N(v)$ contains $D_{2,2}$ or its subdivision. Thus, $N(v)$ contains a path of length $k$, for some $k \geq 2$, say on vertices $v_1, \dots, v_k$, as well as vertices $v_{k+1},v_{k+2}$ adjacent to $v_1$ and $v_{k+3},v_{k+4}$ adjacent to $v_k$. Without loss of generality, $c(vv_i) = i$ for all $i$. For reference, we depict this in Figure~\ref{forbidden4}.

\begin{figure}[h]
\centering

\begin{tikzpicture}

\draw[thick, color = amber] (0,0) to[bend right] (-2,-3);
\draw[thick, color = forest] (0,0) to[bend right] (-1,-3);
\draw[thick, color = blue] (0,0) to[bend left] (1,-3);
\draw[thick, color = periwinkle] (0,0) to[bend left] (2,-3);

\draw (4,-3) -- (2,-3);
\draw (-4,-3) -- (-2,-3);

\draw (2.5,-2) -- (2,-3);
\draw (-2.5,-2) -- (-2,-3);

\draw (-2,-3) -- (-1,-3);
\draw (2,-3) -- (1,-3);

\draw[thick, color = cadmium] (0,0) to[bend right] (-2.5,-2);

\draw[thick, color = byzantine] (0,0) to[bend left] (2.5,-2);

\draw[thick, color = red] (0,0) to[bend right] (-4,-3);

\draw[thick, color = brilliantrose] (0,0) to[bend left] (4,-3);

\draw[dotted] (-1,-3) -- (1,-3);

\filldraw (0,0) circle (0.05 cm);

\filldraw (-1,-3) circle (0.05 cm) node[below]{$v_2$};
\filldraw (1,-3) circle (0.05 cm) node[below]{$v_{k-1}$};
\filldraw (-2,-3) circle (0.05 cm) node[below]{$v_1$};
\filldraw (2,-3) circle (0.05 cm) node[below]{$v_k$};

\filldraw (-2.5,-2) circle (0.05 cm) node[below left]{$v_{k+2}$};
\filldraw (-4,-3) circle (0.05 cm)node[below]{$v_{k+1}$};
\filldraw (2.5,-2) circle (0.05 cm) node[below right]{$v_{k+4}$};
\filldraw (4,-3) circle (0.05 cm)node[below]{$v_{k+3}$};

\end{tikzpicture}

\caption{A subdivision of $D_{2,2}$ in $N(v)$} \label{forbidden4}

\end{figure}
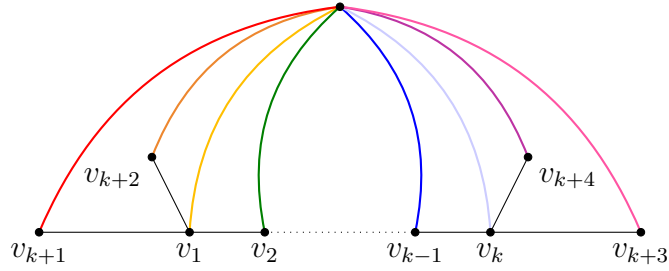

To avoid a rainbow $C_4$-copy, either $c(v_{k+1}v_1) = 2$ or $c(v_1v_2) = k+1$. 
Similarly, either $c(v_{k+2}v_1) = 2$ or $c(v_1v_2) = k+2$. 
Thus, we must have $c(v_1v_2) \in \{k+1, k+2\}$. 
Analogously, we must have $c(v_{k-1}v_k) \in \{k+3, k+4\}$. 
If $k = 2$, we have arrived at a contradiction already. 
If $k > 2$, then observe that to avoid a rainbow $C_4$-copy using $v,v_1,v_2,v_3$, since $c(v_1v_2) \neq 3$, we must have $c(v_2v_3) = 1$.
Continuing in this fashion, for each $i \in \{2, \dots, k-1\}$, we must have $c(v_iv_{i+1}) = i - 1$. 
However, we have already seen that $c(v_{k-1}v_k) \neq k-2$. 
Thus, it is impossible to color any subdivision of $D_{2,2}$ in $N(v)$ to avoid a rainbow $C_4$-copy.\qedhere
\end{enumerate}
\end{proof}

With the above lemma established, we can quickly prove that the following construction is properly rainbow $C_4$-saturated.

\begin{const}\label{const}
Suppose $n \equiv i \mod 6$ with $n \geq 7$.
For convenience, if $n$ is divisible by $6$, we shall set $i = 6$, not $i = 0$. 
Let $G_n$ be the graph consisting of a universal vertex $u$ whose neighborhood induces $\lfloor \frac{n-1}{6} \rfloor - 1$ copies of $S_{1,2,2}$ and one copy of $S_{1, \lceil \frac{6+i-2}{2} \rceil, \lfloor \frac{6+i-2}{2} \rfloor}$. 
\end{const}

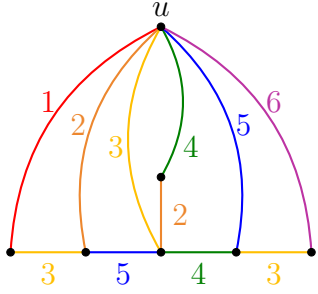
\begin{figure}[h]
\centering

\begin{tikzpicture}

\draw (0,0) node[above]{$u$};

\draw[thick, color = red] (0,0) to[bend right] (-2,-3);
\draw[thick, color = red] (-1.5,-1) node{1};
\draw[thick, color = cadmium] (0,0) to[bend right] (-1,-3);
\draw[thick, color = cadmium] (-1.1,-1.3) node{2};
\draw[thick, color = amber] (0,0) to[bend right] (0,-3);
\draw[thick, color = amber] (-0.6,-1.6) node{3};

\draw[thick, color = forest] (0,0) to[bend left] (0,-2);
\draw[thick, color = forest] (0.4,-1.6) node{4};
\draw[thick, color = blue] (0,0) to[bend left] (1,-3);
\draw[thick, color = blue] (1.1,-1.3) node{5};
\draw[thick, color = byzantine] (0,0) to[bend left] (2,-3);
\draw[thick, color = byzantine] (1.5,-1) node{6};

\draw[thick, color = cadmium] (0,-2) -- (0,-3) node[pos = 0.5, right]{2};

\draw[thick, color = amber] (-2,-3) -- (-1,-3) node[pos = 0.5, below]{3};
\draw[thick, color = amber] (2,-3) -- (1,-3) node[pos = 0.5, below]{3};

\draw[thick, color = blue] (-1,-3) -- (0,-3) node[pos = 0.5, below]{5};
\draw[thick, color = forest] (0,-3) -- (1,-3) node[pos = 0.5, below]{4};

\filldraw (0,0) circle (0.05 cm);

\filldraw (0,-2) circle (0.05 cm);
\filldraw (0,-3) circle (0.05 cm);
\filldraw (-1,-3) circle (0.05 cm);
\filldraw (1,-3) circle (0.05 cm);
\filldraw (-2,-3) circle (0.05 cm);
\filldraw (2,-3) circle (0.05 cm);

\end{tikzpicture}

\caption{A copy of $S_{1,2,2}$ in $N(u)$, colored to avoid a rainbow $C_4$-copy} \label{C4construction}

\end{figure}



\mainupper*

\begin{proof}
Observe that Construction~\ref{const} has $\frac{11}{6} n + O(1)$ edges.
Observe also that in Construction~\ref{const}, every copy of $C_4$ intersects precisely one component of the graph induced on $N(u)$.
Thus, to verify that a coloring of Construction~\ref{const} is rainbow-$C_4$-free, it suffices to verify that each component $C$ of the graph induced on $N(u)$ can be colored so that $\{u\} \cup V(C)$ is rainbow-$C_4$-free.
In Figure~\ref{C4construction}, given a copy of $S_{1,2,2}$ in $N(u)$ and colors for those edges incident to $u$, we exhibit a proper coloring of $S_{1,2,2}$ so that the graph induced on $\{u\} \cup V(S_{1,2,2})$ is rainbow-$C_4$-free. 
We can similarly, for any $i \in \{1, \dots, 6\}$, color a copy of $S_{1, \lceil \frac{6+i-2}{2} \rceil, \lfloor \frac{6+i-2}{2} \rfloor}$ in $N(u)$ so as to avoid a rainbow $C_4$-copy. 
Thus, Construction~\ref{const} admits a rainbow-$C_4$-free proper edge coloring.

We now show that Construction~\ref{const} is rainbow $C_4$-saturated. Since $u$ is a universal vertex, any edge added to the construction is contained within $N(u)$. Label the components of the graph induced on $N(u)$ as $C_1, C_2, \dots, C_k$. If an edge is added between $C_i$ and $C_j$, then, since $C_i$ and $C_j$ each contain a vertex of degree $3$, this added edge will either create a subdivision of $D_{2,2}$ within $N(u)$ or create a vertex $v \in N(u)$ with $|N(u) \cap N(v)| = 4$. By Lemma~\ref{traps} and Proposition~\ref{observations}, either outcome implies that the resulting graph does not admit a rainbow-$C_4$-free proper edge coloring.  If an edge is added between non-adjacent vertices of a single component $C_i$, then we can quickly verify that this addition will create either a triangle with pendant edges from two vertices, a copy of $C_4$, or a copy of $C_k$ with a pendant edge for some $k \geq 5$. In any case, by Lemma~\ref{traps}, the addition of an edge to $N(u)$ must yield a graph not admitting a rainbow-$C_4$-free proper edge coloring.
\end{proof}

We now work to show that Construction~\ref{const} is asymptotically best possible. To do so, we require some general understanding of the structure of rainbow $C_4$-saturated graphs. We begin with the following lemma, which shows that any rainbow $C_4$-saturated graph on sufficiently many vertices contains a dominating set which is either small or dense.

\begin{lemma}\label{dominating sets}

Fix $0 < \alpha < 1/9$. There exists $n_0 \in \mathbb{N}$ such that for any $n$-vertex, rainbow $C_4$-saturated graph $G$ with $n \geq n_0$, $G$ contains a dominating set $D$ such that either 

\begin{enumerate}
\item $|D| \leq \alpha n$, 

or 

\item $G[D]$ has average degree at least $4$.
\end{enumerate}

\end{lemma}

\begin{proof}

  Fix $n_0 \geq \frac{6}{\alpha}\binom{11/\alpha^2}{2}$, and let $G$ be a rainbow $C_4$-saturated graph on $n \geq n_0$ vertices. By Proposition~\ref{observations} Case~\ref{obs:dist}, any two vertices $x,y \in V(G)$ must have $d(x,y) \leq 3$. 
Also note that if $\delta(G) \geq 4$, then $V(G)$ trivially forms a dominating set satisfying condition $2$. So we may assume that there exists $v \in V(G)$ with $d(v) \leq 3$, and for any $u \in V(G) \setminus\{v\}$, we have $d(u,v) \leq 3$. Denote by $L_i$ the set of vertices at distance exactly $i$ from $v$. Observe that both $L_1 \cup L_2$ and $L_1 \cup L_3$ are dominating sets, so we are done if one of $L_1 \cup L_2, L_1 \cup L_3$ satisfies one of conditions $1,2$. 

First, note that $L_1 \cup L_2$ will satisfy condition $1$ if $|L_2| \leq \alpha n - 3$. In particular, if each neighbor of $v$ has degree at most $\frac{\alpha}{3}n - 1$, then $L_1 \cup L_2$ satisfies condition $1$. Thus, we may assume that $v$ is adjacent to a vertex of degree at least $\frac{\alpha}{3}n$. In fact, by analogous reasoning, for a fixed constant $k$, we may assume that any vertex $u$ with $d(u) \leq k$ has a neighbor of degree at least $\frac{\alpha}{k}n$. Setting $k = \frac{11}{\alpha}$, each vertex of degree at most $\frac{11}{\alpha}$ has a neighbor of degree at least $\frac{\alpha^2}{11}n$. 

Now, consider $L_3$.
We may assume $|L_3| \geq \alpha n - 3$, else $L_1 \cup L_3$ satisfies condition $1$. This bound on $|L_3|$ implies that if $G[L_3]$ has average degree at least $5$, then $G[L_1 \cup L_3]$ has average degree at least $4$ and thus $L_1 \cup L_3$ satisfies condition $2$. 
So we may assume that $G[L_3]$ has average degree smaller than $5$. 
We shall use this assumption to find a small dominating set. 
First, note that if $u \in L_3$, then $N(u) \subset L_2 \cup L_3$. 
Moreover, by Proposition~\ref{observations}, common neighborhoods in $G$ are of size at most $3$. 
Thus, since $L_2$ is spanned by at most three neighborhoods, we have $|N(u) \cap L_2| \leq 9$. 
In particular, if $u \in L_3$ has $d(u) \geq \frac{11}{\alpha}$, then $u$ has at least $\frac{11}{\alpha} - 9 \geq \frac{10}{\alpha}$ neighbors in $L_3$ (since $\alpha < 1/9$). Thus, if more than $\frac{\alpha}{2} |L_3|$ vertices in $L_3$ have degree at least $\frac{11}{\alpha}$, then $G[L_3]$ has average degree at least $\frac{10}{\alpha} \cdot \frac{\alpha}{2} = 5$. So $L_3$ contains fewer than $\frac{\alpha}{2}|L_3| < \frac{\alpha}{2}n$ vertices of degree at least $\frac{11}{\alpha}$. We define
\[H \coloneqq \left\{u \in L_3 : d(u) \geq \frac{11}{\alpha}\right\},\]
noting that $H$ may be empty. Now, for $u \in L_3 \setminus H$, $u$ must have a neighbor of degree at least $\frac{\alpha^2}{11}n$. Let 
\[S \coloneqq \left\{x \in V(G) : d(x) \geq \frac{\alpha^2}{11}n \text{ and } N(x) \cap \left(L_3 \setminus H \right) \neq \emptyset \right\}.\]
We shall construct a dominating set using $S$. The manner in which we do so depends upon $|S|$.

\begin{enumerate}

\item $|S| \leq \frac{\alpha}{2}n - 3$ 

In this case, note that $S$ dominates $L_3 \setminus H$, so $S \cup H \cup L_1$ is a dominating set in $G$. Since $|L_1| \leq 3$ and $|H| \leq \frac{\alpha}{2}n$, we have $|S \cup H \cup L_1| \leq \alpha n$, satisfying condition $1$.

\item $|S| \geq  \frac{\alpha}{2}n - 3$

In this case, by choice of $n_0$, we have $|S| \geq \frac{11}{\alpha^2}$. Choose $S' \subset S$ with $|S'| = \frac{11}{\alpha^2}$. Since every vertex in $S'$ has degree at least $\frac{\alpha^2}{11}n$, we have 

\[ \left| \bigcup_{x \in S'} N(x)\right| \geq \frac{11}{\alpha^2} \cdot \frac{\alpha^2}{11}n - \sum_{x \neq y \in S'} |N(x) \cap N(y)|\]
by inclusion-exclusion. Since $|S'| = \frac{11}{\alpha^2}$ and $|N(x) \cap N(y)| \leq 3$ for any $x,y \in V(G)$ (by \Cref{observations} Case~\ref{obs:nighbr}), we have
\[ \left|\bigcup_{x \in S'} N(x)\right| \geq n - 3\binom{11/\alpha^2}{2}.\]
Thus, we can create a dominating set in $G$ by adding at most $3\binom{11/\alpha^2}{2}$ vertices to $S'$. By choice of $n_0$, the resulting dominating set is of size at most $\alpha n$, satisfying condition~$1$.
\end{enumerate}
\end{proof}

Using Lemma~\ref{dominating sets}, we can now show that the upper bound given in Theorem~\ref{c4 upper bound} is asymptotically best possible. 




\mainlower*

\begin{proof}

Given $\varepsilon > 0$, we choose $n_0$ as given by applying Lemma~\ref{dominating sets} with $\alpha = \frac{5 \varepsilon}{11}$. Let $G$ be a rainbow $C_4$-saturated graph on vertex set $V$, with $|V| = n  \geq n_0$, and let $D$ be a dominating set as guaranteed by Lemma~\ref{dominating sets}. Thus, either $|D| < \frac{5 \varepsilon}{11} n$ or $G[D]$ has average degree at least~$4$. We shall estimate $|E(G)|$ by considering $G[V \setminus D]$. For a component $C$ of $G[V \setminus D]$, let $f(C)$ be the number of edges incident to $C$. Observe that any component $C$ is incident to at least $|V(C)| -1$ edges within $G[V \setminus D]$ and at least $|V(C)|$ edges which are incident to a vertex of $D$. We shall call a component $C$ of $G[V \setminus D]$ \textit{sparse} if $f(C) = 2|V(C)| - 1$, and \textit{dense} if $f(C) \geq 2|V(C)|$. We let $\mathcal{C}$ be the set of components of $G[V \setminus D]$, and partition $\mathcal{C} = \mathcal{C}_d \cup \mathcal{C}_s$, where $\mathcal{C}_d$, $\mathcal{C}_s$ are the sets of dense and sparse components, respectively, of $G[V\setminus D]$. We thus have 

\begin{align}
|E(G)| &= |E(G[D])| + \sum_{C \in \mathcal{C}} f(C)  \geq |E(G[D])| + \sum_{C \in \mathcal{C}_d} 2|V(C)| + \sum_{C \in \mathcal{C}_s} 2|V(C)| - 1 \nonumber\\
       &\geq |E(G[D])| + 2|V \setminus D| - |\mathcal{C}_s|\label{eq:edgeDensity}
\end{align}

We now bound $|\mathcal{C}_s|$ by showing that sparse components are usually not too small. Our eventual goal is to bound the number of components $C \in \mathcal{C}_s$ with $|V(C)| < 6$. We begin by making some elementary observations about the elements of $\mathcal{C}_s$, which will be repeatedly useful. Note that if $C \in \mathcal{C}_s$, then there are precisely $|V(C)|$ edges between $V(C)$ and $D$; thus, each member of $V(C)$ has exactly one neighbor in $D$. Also note that if $C \in \mathcal{C}_s$, then $|E(C)| = |V(C)| - 1$, and thus $C$ is a tree. Also, by Proposition~\ref{observations} Case~\ref{obs:deg1}, $G$ contains at most one vertex of degree $1$, so for all but at most one component $C \in \mathcal{C}_s$, $C$ contains at least one edge. In particular, all but one $C \in \mathcal{C}_s$ must contain at least two vertices which have degree~$1$ in $G[V \setminus D]$, and, by the previous observation, these vertices have degree exactly $2$ in $G$.

Define 
\[ L \coloneqq \left\{ v \in \bigcup_{C \in \mathcal{C}_s} V(C) : d_G(v) = 2 \right\}. \]
We begin by showing that $D$ contains a small set whose neighborhoods either contain $L$ or contain all sparse components of $G[V \setminus D]$ of size less than $6$.

\begin{claim}\label{core nbhds}
  There exists a set $S \subseteq D$ with $|S| \leq 35$ such that either
  \begin{enumerate}
    \item $L \subseteq N(S)$,
    
    or
    
    \item for every $C \in \mathcal{C}_s$ with $V(C) \not\subseteq N(S)$, $|C\cap N(S)|\ge 5$.
  \end{enumerate}
\end{claim}

\begin{proof}[Proof of Claim~\ref{core nbhds}]
  \let\qed\relax

  We describe a procedure for building such a set $S$. Fix $v_1,v_2 \in D$ such that there exist $w_1 \in N(v_1)\setminus D, w_2 \in N(v_2)\setminus D$ with $w_1,w_2 \in L$. (If such a pair $v_1,v_2$ does not exist, then, since $L\subseteq \cup_{C \in \mathcal{C}_s}V(C)$, either $\mathcal{C}_s$ is empty and we return $S = \emptyset$ or there is a single vertex $v \in D$ such that $L \subseteq N(v)$ and we return $S = \{v\}$.) We set $S_1\subseteq D$ as:
\[S_1\coloneqq \Big(\{v_1,v_2\} \cup \left(N(v_1)\cap N(v_2)\right)\cup N\bigl(N(\{w_1,w_2\})\setminus D\bigr)\Big)\cap D.\]

We claim that every component $C \in \mathcal{C}_s$ is either contained in $N(S_1)$, or else intersects one of $N(v_1), N(v_2)$.
Indeed, suppose $C \in \mathcal{C}_s$ and $x \in V(C)$ such that $x \not \in N(S_1)$. That is, the unique element $u \in N(x) \cap D$ has $u \not\in S_1$.
We must show that $C$ intersects $N(\{v_1,v_2\})$. Indeed, suppose that it does not. Then $w_1,w_2 \not\in V(C)$ as well as $x$ has no neighbor in $N(\{v_1,v_2\}) \setminus D$, and in particular, $xw_1, xw_2 \not\in E(G)$. Note that, since $w_1,w_2 \not\in V(C)$, the addition of either $xw_1$ or $xw_2$ to $E(G)$ creates no copy of $C_4$ in $G[V \setminus D]$. Moreover, since $N(\{w_1,w_2\}) \cap (N(u) \setminus D) = \emptyset$, the addition of either edge cannot create a copy of $C_4$ using three vertices from $V \setminus D$ and $u$. Similarly, since $x$ has no neighbor in $N(\{v_1,v_2\}) \setminus D$, the addition of either edge cannot create a copy of $C_4$ using three vertices from $V\setminus D$ and either $v_1$ or $v_2$. Thus, the only possible copy of $C_4$ created by adding $xw_i$ uses two vertices from $D$, and these must be $u$ and $v_i$, the only neighbour in $D$ of $x$ and $w_i$, respectively. Since $u \not\in N(v_1)\cap N(v_2)$, we can therefore add one of $xw_1,xw_2$ without creating a copy of $C_4$, a contradiction. 




Since $|N(v_1)\cap N(v_2)| \leq 3$ by Proposition~\ref{observations} Case~\ref{obs:nighbr}, and since $w_1,w_2$ each have exactly one neighbor in $V\setminus D$, we have $|S_1| \leq 7$. As we have argued, any component $C \in \mathcal{C}_s$ which is not contained in $N(S_1)$ must intersect $N(S_1)$ once, either in $N(v_1)$ or $N(v_2)$.

Now, we wish to repeat this argument by selecting $v_3,v_4 \not\in S_1$ such that $N(v_3)\setminus D$ and $N(v_4)\setminus D$ contain vertices $w_3,w_4 \in L$, respectively. If it is not possible to do so, then either $L \subseteq N(S_1)$,  or there exists precisely one vertex $v \in D\setminus S_1$ such that $N(S_1 \cup \{v\}) \supseteq L$. In these cases, we either return $S = S_1$ or $S = S_1 \cup \{v\}$. If we can find $v_3,v_4$ as desired, we set $S_2\subseteq D$ as:  
\[S_2\coloneqq  \Big(S_1 \cup \{v_3,v_4\} \cup \left(N(v_3)\cap N(v_4)\right)\cup N\bigl(N(\{w_3,w_4\})\setminus D\bigr)\Big)\cap D.\]
%
%

Repeating the above argument, any $C \in \mathcal{C}_s$ which is not contained in $N(S_2)$ must contain a vertex from either $N(v_3) \setminus D$ or $N(v_4)\setminus D$, and thus intersects $N(S_2)$ at least twice. (We again use the fact that, since $C$ is sparse, each vertex of $C$ has exactly one neighbor in $D$; thus, if $C$ intersects $N(v_i)$ and $N(v_j)$, it must do so in two distinct vertices.) Note also that $|S_2| \leq 14$. Proceeding in this fashion, we either return a set $S$ of size at most $7\cdot 4 + 1 = 29$ with $L \subseteq N(S)$, or we return $S_5$ containing disjoint pairs $\{v_{2i-1}, v_{2i}\}$ ($k \in \{1, \dots, 5\}$) with the property that any $C \in \mathcal{C}_s$ which is not contained in $ N(S_5)$ intersects one of $N(v_{2i - 1})\setminus D, N(v_{2i}) \setminus D$ for each $i \in \{1,\dots,5\}$. Here, we note that $|S_5| \leq 35$, and set $S = S_5$. 
\hfill $\diamondsuit$
\end{proof}

We shall call the elements of $S$ the \textit{core vertices} of $D$, and say that the sets $N(v)\setminus D$ with $v \in S$ are the \textit{core neighborhoods}. As we shall argue, to understand the sizes of components $C \in \mathcal{C}_s$, it essentially suffices to understand the sizes of components in the subgraphs of $G$ induced on the core neighborhoods. Towards such an argument, we also require the following claim, which states that for any $v \in D$, few members of $L \cap N(v)$ are in very small components of $G[V \setminus D]$.

\begin{claim}\label{L comps}
Fix $v \in D$ and let $\mathcal{C}_v$ be the set of components of $G[N(v) \setminus D]$ which intersect $L$. Then at most two components $C \in \mathcal{C}_v$ have $|V(C)| < 3$. Moreover, let $\mathcal{C}_v'$ be the set of components in $\mathcal{C}_v$ which are also components of $G[V \setminus D]$. Then no component $C \in \mathcal{C}_v'$ has $3 \leq |V(C)| < 6$.

\end{claim}

\begin{proof}[Proof of Claim~\ref{L comps}]
  \let\qed\relax

First, we show that there cannot be two components $C_1, C_2 \in \mathcal{C}_v$ of size $1$. Indeed, suppose $w_1, w_2 \in L$ are isolated vertices in $G[N(v) \setminus D]$. We show that we can add the edge $w_1w_2$ to $G$ without forcing a rainbow $C_4$-copy, a contradiction since $G$ is rainbow $C_4$-saturated. We have two cases.

\begin{enumerate} 

  \item $|N(w_1)\cap N(w_2)| = 2$

In this case, observe that the addition of $w_1w_2$ to $G$ creates no new $C_4$-copies; thus, we may add $w_1w_2$ freely.

\item $|N(w_1)\cap N(w_2)| = 1$

In this case, let $u_1,u_2$ be the neighbors of $w_1,w_2$, respectively, in $G \setminus D$. Now, observe that the addition of $w_1w_2$ will create no copy of $C_4$ if $u_1u_2 \not\in E(G)$, and will create precisely one copy of $C_4$ if $u_1u_2 \in E(G)$, namely $w_1 w_2 u_2 u_1 w_1$. 
Thus, we suppose that $u_1u_2 \in E(G)$. Consider a rainbow $C_4$-free edge-coloring of $G$. 
We begin by maintaining this edge-coloring of $G$ while adding the edge $w_1w_2$, and seek to select $c(w_1w_2)$ so that the resulting edge-coloring of $G + w_1w_2$ is proper and $w_1 w_2 u_2 u_1 w_1$ is not rainbow. 

Note first that if $w_1 u_1 u_2 w_2$ is not a rainbow copy of $P_3$, then we may choose $c(w_1w_2)$ freely. Thus, we may assume that $w_2 u_2 u_1 w_1$ \textit{is} a rainbow copy of $P_3$, say with $c(w_1u_1) = 1, c(u_1u_2) = 2$, and $c(u_2w_2) = 3$. Thus, we wish to choose $c(w_1w_2) = 2$. If we cannot do so, then without loss of generality, $c(w_1v) = 2$. We now show that if $c(w_1v) = 2$, then we can change $c(w_1v)$ without creating a rainbow $C_4$-copy, allowing the selection $c(w_1w_2) = 2$. Clearly, if $w_1v$ is not in any $C_4$-copy, then we may freely recolor $w_1v$. If $w_1v$ is in a $C_4$ copy, then note that this $C_4$ copy must span $v, w_1,u_1$ and a neighbor of $w_1,u_1$, say $x$. We illustrate the configuration in Figure~\ref{uniquec4}, in addition to known adjacencies and colors between $v, w_1,w_2,u_1,$ and $u_2$.

\begin{figure}[h]

\begin{center}
\begin{tikzpicture}

\draw[thick, color = red]  (-1,-2) -- (-1,-4) node[pos = 0.5, left]{$1$};  

\draw[thick, color = cadmium]  (-1,-4) -- (1,-4) node[pos = 0.5, below]{$2$};

\draw[thick, color = amber]  (1,-4) -- (1,-2) node[pos = 0.5, right]{$3$}; 

\draw[thick, color = cadmium]  (0,0) -- (-1,-2) node[pos = 0.5, left]{$2$}; 

\filldraw (0,0) circle (0.05 cm) node[above]{$v$};

\filldraw (-1,-2) circle (0.05 cm) node[left]{$w_1$};
\filldraw (1,-2) circle (0.05 cm) node[right]{$w_2$};

\filldraw (-1,-4) circle (0.05 cm) node[below]{$u_1$};
\filldraw (1,-4) circle (0.05 cm) node[below]{$u_2$};

\filldraw (-3, -2) circle (0.05 cm) node[left]{$x$};

\draw (0,0) -- (1, -2);
\draw (0,0) to[bend right] (-3,-2) to[bend right] (-1,-4);

\end{tikzpicture}
\caption{A copy of $C_4$ using $w_1v$}\label{uniquec4}
\end{center}

\end{figure}
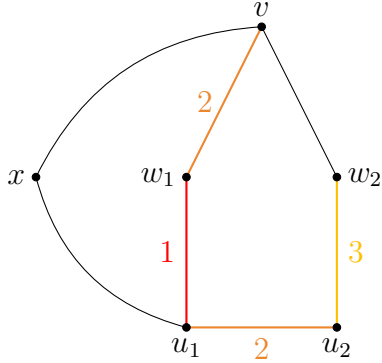

Note that $c(u_1x) \neq 2$, so since $vw_1u_1xv$ is not rainbow, we must have $c(vx) = 1$. Thus, $vw_1u_1xv$ is in fact the unique copy of $C_4$ containing $vw_1$, and will not be made rainbow by altering $c(vw_1)$. Therefore, if $c(vw_1) = 2$, then we alter the coloring of $G$ by setting $c(vw_1)$ to a new color, and we are able to add $w_1w_2$ with $c(w_1w_2) = 2$.

\medskip
\end{enumerate}
Next, we show that there cannot be two components $C_1, C_2 \in \mathcal{C}_v$ of size $2$. Indeed, suppose that two such components exist, and say $V(C_1) = \{w_1,u_1\}$ and $V(C_2) = \{w_2,u_2\}$, where $w_1,w_2 \in L$. Note that although $u_1,u_2$ may have high degree in $G$, each has exactly one neighbor in $N(v)$.

Observe that neither $vw_1$ nor $w_1u_1$ is contained in any $C_4$-copy in $G$. Indeed, since $d(w_1) = 2$, a $C_4$-copy in $G$ containing $w_1$ would necessarily include both $v$ and $u_1$, with a fourth vertex in $N(v)\cap N(u_1)$. However, since $N(v)\cap N(u_1) = \{w_1\}$, it is impossible to select this fourth vertex. Analogously, neither $vw_2$ nor $w_2u_2$ is contained in any $C_4$-copy in $G$.

Now, we claim that we can add the edge $w_1w_2$ to $G$ without forcing a rainbow $C_4$-copy. Note that the addition of $w_1w_2$ creates precisely two $C_4$-copies: $vw_1w_2u_2v$ and $vw_2w_1u_1v$. We shall color $G + w_1w_2$ as follows. Begin with a rainbow $C_4$-free proper edge-coloring of $G$. Since $vw_1, vw_2, w_1u_1,$ and $w_2u_2$ are not contained in any $C_4$-copies in $G$, we can modify this edge-coloring of $G$ as follows: choose colors $a,b$ which do not appear in the chosen edge-coloring of $G$. We re-color $vw_1$ and $w_2u_2$ with color $a$ and $vw_2, w_1u_1$ with color $b$. This modification cannot create a rainbow $C_4$-copy in $G$, since the re-colored edges do not appear in any $C_4$-copies in $G$; moreover, the new edge-coloring of $G$ remains proper by the selection of $a,b$. Now, we add $w_1w_2$ to $G$, choosing $c(w_1,w_2) \not\in \{a,b\}$. Note that neither $vw_1w_2u_2v$ nor $vw_2w_1u_1v$ is rainbow under the chosen coloring.

Finally, we show that there are no components in $\mathcal{C}_v'$ on $3,4,$ or $5$ vertices. 
Observe, for $C \in \mathcal{C}_v'$, the only copies of $C_4$ which intersect $V(C)$ are contained in $G[V(C) \cup \{v\}]$. 
Thus, we may add edges between vertices in $V(C)$ and change the colors of edges in $E(C)$ in any fashion we like, so long as the resulting changes do not result in a rainbow $C_4$-copy in $G[V(C) \cup \{v\}]$. 
We consider all possible components $C \in \mathcal{C}_v'$ with $|V(C)| \in \{3,4,5\}$, and show that we may add an edge $e$ to each and color $C + e$ to avoid a rainbow $C_4$-copy in $G[V(C) \cup \{v\}]$. 
In each case, we shall assume that $V(C) = \{v_1, \dots, v_i\}$, and that $c(vv_i) = i$.

\begin{enumerate}
\item $|V(C)| = 3$.

In this case, $C$ is the unique $3$-vertex tree, $P_3$. Say $v_1,v_3$ are the leaves of $C$. We may add the edge $v_1v_3$ and color $C + v_1v_3$ to avoid a rainbow $C_4$-copy. We depict this addition and an admissible coloring in Figure~\ref{case 1}.

\begin{figure}[h]
\begin{center}

\begin{tikzpicture}

\draw[thick, color = red] (0,0) -- (-2,-2);
\draw[thick, color = cadmium] (0,0) -- (0,-2);
\draw[thick, color = amber] (0,0) -- (2,-2);
\draw[thick, color = red] (-1.3,-1) node{$1$};
\draw[thick, color = cadmium] (0,-1)node[left]{$2$};
\draw[thick, color = amber] (1.3,-1) node{$3$};

\filldraw (0,0) circle (0.05 cm) node[above]{$v$};

\filldraw (-2,-2) circle (0.05 cm) node[below left]{$v_1$};
\filldraw (0,-2) circle (0.05 cm) node[below]{$v_2$};
\filldraw (2,-2) circle (0.05 cm) node[below right]{$v_3$};
\draw (-2,-2) -- (2,-2);

\draw[thick, color = red] (6,0) -- (4,-2);
\draw[thick, color = cadmium] (6,0) -- (6,-2);
\draw[thick, color = amber] (6,0) -- (8,-2);
\draw[thick, color = red] (4.7,-1) node{$1$};
\draw[thick, color = cadmium] (6,-1)node[left]{$2$};
\draw[thick, color = amber] (7.3,-1) node{$3$};

\draw[thick, color = amber] (4,-2) -- (6,-2) node[pos = 0.5,below]{$3$};
\draw[thick, color = red] (6,-2) -- (8,-2) node[pos = 0.5,below]{$1$};
\draw[thick, color = cadmium] (4,-2) to[bend right = 50] (8,-2);
\draw[thick, color = cadmium] (6,-3.1)node{$2$};

\filldraw (6,0) circle (0.05 cm) node[above]{$v$};

\filldraw (4,-2) circle (0.05 cm) node[below left]{$v_1$};
\filldraw (6,-2) circle (0.05 cm) node[below]{$v_2$};
\filldraw (8,-2) circle (0.05 cm) node[below right]{$v_3$};

\end{tikzpicture}

\caption{$C$ and $C + v_1v_3$, colored to avoid a rainbow $C_4$-copy}\label{case 1}
\end{center}
\end{figure}

\item $|V(C)| = 4$, and $C = P_4$

In this case, we label the vertices of $C$ in a natural linear order, so $v_1,v_4$ are the endpoints of the path. We may add, for example, $v_1v_3$, and color $C + v_1v_3$ to avoid a rainbow $C_4$-copy. We depict this addition and an admissible coloring in Figure~\ref{case 2}.

\begin{figure}[h]
\begin{center}

\begin{tikzpicture}

\draw[thick, color = red] (-1,0) to[bend right](-4,-2);
\draw[thick, color = cadmium] (-1,0) to[bend right](-2,-2);
\draw[thick, color = amber] (-1,0) to[bend left](0,-2);
\draw[thick, color = forest] (-1,0) to[bend left](2,-2);
\draw[thick, color = red] (-3.5,-0.9) node{$1$};
\draw[thick, color = cadmium] (-2,-0.9) node{$2$};
\draw[thick, color = amber] (0,-0.9) node{$3$};
\draw[thick, color = forest] (1.5,-0.9) node{$4$};

\filldraw (-1,0) circle (0.05 cm) node[above]{$v$};
\filldraw (-4,-2) circle (0.05 cm) node[below left]{$v_1$};
\filldraw (-2,-2) circle (0.05 cm) node[below]{$v_2$};
\filldraw (0,-2) circle (0.05 cm) node[below]{$v_3$};
\filldraw (2,-2) circle (0.05 cm) node[below right]{$v_4$};
\draw (-4,-2) -- (2,-2);

\draw[thick, color = red] (4.5,-0.9) node{$1$};
\draw[thick, color = cadmium] (6,-0.9) node{$2$};
\draw[thick, color = amber] (8,-0.9) node{$3$};
\draw[thick, color = forest] (9.5,-0.9) node{$4$};

\draw[thick, color = red] (7,0) to[bend right](4,-2);
\draw[thick, color = cadmium] (7,0) to[bend right](6,-2);
\draw[thick, color = amber] (7,0) to[bend left](8,-2);
\draw[thick, color = forest] (7,0) to[bend left](10,-2);
\draw[thick, color = amber] (4,-2) -- (6,-2) node[pos = 0.5, below]{$3$};
\draw[thick, color = red] (6,-2) -- (8,-2) node[pos=0.5, below]{$1$};
\draw[thick, color = cadmium] (8,-2) -- (10,-2) node[pos=0.5, below]{$2$};
\draw[thick, color = forest] (4,-2) to[bend right = 50] (8,-2);
\draw[thick, color = forest] (6, -3.1) node{$4$};

\filldraw (7,0) circle (0.05 cm) node[above]{$v$};

\filldraw (4,-2) circle (0.05 cm) node[below left]{$v_1$};
\filldraw (6,-2) circle (0.05 cm) node[below]{$v_2$};
\filldraw (8,-2) circle (0.05 cm) node[below right]{$v_3$};
\filldraw (10,-2) circle (0.05 cm) node[below right]{$v_4$};

\end{tikzpicture}

\caption{$C$ and $C + v_1v_3$, colored to avoid a rainbow $C_4$-copy}\label{case 2}
\end{center}
\end{figure}

\item $|V(C)| = 4$ and $C = S_3$

In this case, $C$ is a star, say with center $v_1$ and leaves $v_2,v_3,v_4$. We may add any edge between leaves, say $v_3v_4$, and color $C + v_3v_4$ to avoid a rainbow $C_4$-copy. We depict this addition and an admissible coloring in Figure~\ref{case 3}.

\begin{figure}[h]
\begin{center}

\begin{tikzpicture}

\draw[thick, color = red] (-1,0) to[bend right](-4,-2);
\draw[thick, color = cadmium] (-1,0) to[bend right](-2,-2);
\draw[thick, color = amber] (-1,0) to[bend left](0,-2);
\draw[thick, color = forest] (-1,0) to[bend left](2,-2);
\draw[thick, color = red] (-3.5,-0.9) node{$1$};
\draw[thick, color = cadmium] (-2,-0.9) node{$2$};
\draw[thick, color = amber] (0,-0.9) node{$3$};
\draw[thick, color = forest] (1.5,-0.9) node{$4$};

\filldraw (-1,0) circle (0.05 cm) node[above]{$v$};
\filldraw (-4,-2) circle (0.05 cm) node[below left]{$v_1$};
\filldraw (-2,-2) circle (0.05 cm) node[below]{$v_2$};
\filldraw (0,-2) circle (0.05 cm) node[below]{$v_3$};
\filldraw (2,-2) circle (0.05 cm) node[below right]{$v_4$};
\draw (-4,-2) -- (-2,-2);
\draw (-4,-2) to[bend right] (0,-2);
\draw (-4,-2) to[bend right = 50] (2,-2);

\draw[thick, color = red] (4.5,-0.9) node{$1$};
\draw[thick, color = cadmium] (6,-0.9) node{$2$};
\draw[thick, color = amber] (8,-0.9) node{$3$};
\draw[thick, color = forest] (9.5,-0.9) node{$4$};

\draw[thick, color = red] (7,0) to[bend right](4,-2);
\draw[thick, color = cadmium] (7,0) to[bend right](6,-2);
\draw[thick, color = amber] (7,0) to[bend left](8,-2);
\draw[thick, color = forest] (7,0) to[bend left](10,-2);

\draw[thick, color = amber] (4,-2) -- (6,-2) node[pos = 0.5, above]{$3$};
\draw[thick, color = forest] (4,-2) to[bend right] (8,-2);
\draw[thick, color = cadmium] (4,-2) to[bend right = 50] (10,-2);
\draw[thick, color = red] (8,-2) -- (10,-2) node[pos = 0.5, above]{$1$};
\draw[thick, color = forest] (6,-2.8) node{$4$};
\draw[thick, color = cadmium] (7, -3.6) node{$2$};

\filldraw (7,0) circle (0.05 cm) node[above]{$v$};

\filldraw (4,-2) circle (0.05 cm) node[below left]{$v_1$};
\filldraw (6,-2) circle (0.05 cm) node[below]{$v_2$};
\filldraw (8,-2) circle (0.05 cm) node[below right]{$v_3$};
\filldraw (10,-2) circle (0.05 cm) node[below right]{$v_4$};

\end{tikzpicture}

\caption{$C$ and $C + v_3v_4$, colored to avoid a rainbow $C_4$-copy}\label{case 3}
\end{center}
\end{figure}
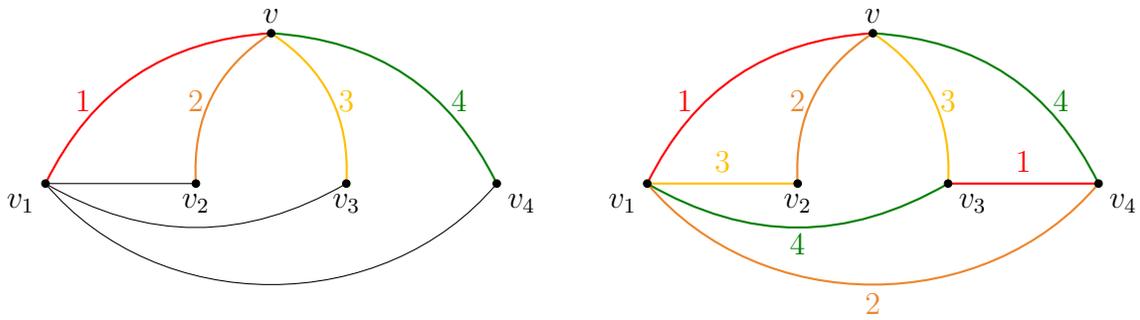

\item $|V(C)| = 5$, and $C$ has diameter $4$

In this case, $C$ is the $5$-vertex path $P_5$. Say $v_1,v_5$ are the leaves of $C$. We may add the edge $v_1v_5$ and color $C + v_1v_5$ to avoid a rainbow $C_4$-copy. We depict this addition and an admissible coloring in Figure~\ref{case 4}.

\begin{figure}[h]
\begin{center}

\begin{tikzpicture}

\draw[thick, color = red] (-2,0) to[bend right](-6,-2);
\draw[thick, color = cadmium] (-2,0) to[bend right](-4,-2);
\draw[thick, color = amber] (-2,0) to(-2,-2);
\draw[thick, color = forest] (-2,0) to[bend left](0,-2);
\draw[thick, color = blue] (-2,0) to[bend left](2,-2);

\draw (-6,-2) --(2,-2);
\draw[thick, color = red] (-5,-0.5) node{$1$};
\draw[thick, color = cadmium] (-3.5,-0.5) node{$2$};
\draw[thick, color = amber] (-1.8,-0.5) node{$3$};
\draw[thick, color = forest] (-0.5,-0.5) node{$4$};
\draw[thick, color = blue] (1,-0.5) node{$5$};

\filldraw (-2,0) circle (0.05 cm) node[above]{$v$};
\filldraw (-6,-2) circle (0.05 cm) node[below left]{$v_1$};
\filldraw (-4,-2) circle (0.05 cm) node[below]{$v_2$};
\filldraw (-2,-2) circle (0.05 cm) node[below]{$v_3$};
\filldraw (0,-2) circle (0.05 cm) node[below]{$v_4$};
\filldraw (2,-2) circle (0.05 cm) node[below right]{$v_5$};

\draw[thick, color = red] (-2,-4) to[bend right](-6,-6);
\draw[thick, color = cadmium] (-2,-4) to[bend right](-4,-6);
\draw[thick, color = amber] (-2,-4) to(-2,-6);
\draw[thick, color = forest] (-2,-4) to[bend left](0,-6);
\draw[thick, color = blue] (-2,-4) to[bend left](2,-6);

\draw[thick, color = red] (-4,-6) -- (-2,-6) node[pos = 0.5, above]{$1$};
\draw[thick, color = cadmium] (-2,-6) -- (0,-6) node[pos = 0.5, above]{$2$};
\draw[thick, color = amber] (0,-6) -- (2,-6) node[pos = 0.5, above]{$3$};
\draw[thick, color = forest] (2,-6) to[bend left = 30] (-6,-6);
\draw[thick, color = forest] (-2,-6.9) node{$4$};
\draw[thick, color = blue] (-6,-6) -- (-4,-6)node[pos = 0.5, above]{$5$};

\draw[thick, color = red] (-5,-4.5) node{$1$};
\draw[thick, color = cadmium] (-3.5,-4.5) node{$2$};
\draw[thick, color = amber] (-1.8,-4.5) node{$3$};
\draw[thick, color = forest] (-0.5,-4.5) node{$4$};
\draw[thick, color = blue] (1,-4.5) node{$5$};

\filldraw (-2,-4) circle (0.05 cm) node[above]{$v$};
\filldraw (-6,-6) circle (0.05 cm) node[below left]{$v_1$};
\filldraw (-4,-6) circle (0.05 cm) node[below]{$v_2$};
\filldraw (-2,-6) circle (0.05 cm) node[below]{$v_3$};
\filldraw (0,-6) circle (0.05 cm) node[below]{$v_4$};
\filldraw (2,-6) circle (0.05 cm) node[below right]{$v_5$};

\end{tikzpicture}

\caption{$C$ and $C + v_1v_5$, colored to avoid a rainbow $C_4$-copy}\label{case 4}
\end{center}
\end{figure}
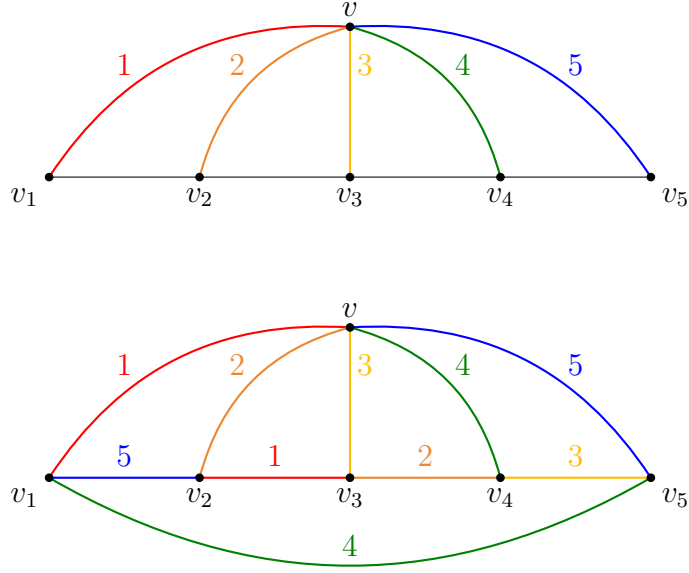

\item $|V(C)| = 5$, and $C$ has diameter $3$

In this case, $C$ is the unique $5$-vertex tree which contains one vertex of degree $3$. Say $v_3$ is the vertex of degree $3$, and $v_4,v_5$ are the leaves of $C$ which are adjacent to $v_3$. We may add the edge $v_4v_5$ and color $C + v_4v_5$ to avoid a rainbow $C_4$-copy. We depict this addition and an admissible coloring in Figure~\ref{case 5}.

\begin{figure}[h]
\begin{center}

\begin{tikzpicture}

\draw[thick, color = red] (-2,0) to[bend right](-6,-2);
\draw[thick, color = cadmium] (-2,0) to[bend right](-4,-2);
\draw[thick, color = amber] (-2,0) to(-2,-2);
\draw[thick, color = forest] (-2,0) to[bend left](0,-2);
\draw[thick, color = blue] (-2,0) to[bend left](2,-2);

\draw (-6,-2) --(0,-2);
\draw (-2,-2) to[bend right = 50] (2, -2);
\draw[thick, color = red] (-5,-0.5) node{$1$};
\draw[thick, color = cadmium] (-3.5,-0.5) node{$2$};
\draw[thick, color = amber] (-1.8,-0.5) node{$3$};
\draw[thick, color = forest] (-0.5,-0.5) node{$4$};
\draw[thick, color = blue] (1,-0.5) node{$5$};

\filldraw (-2,0) circle (0.05 cm) node[above]{$v$};
\filldraw (-6,-2) circle (0.05 cm) node[below left]{$v_1$};
\filldraw (-4,-2) circle (0.05 cm) node[below]{$v_2$};
\filldraw (-2,-2) circle (0.05 cm) node[below left]{$v_3$};
\filldraw (0,-2) circle (0.05 cm) node[below]{$v_4$};
\filldraw (2,-2) circle (0.05 cm) node[below right]{$v_5$};

\draw[thick, color = red] (-2,-4) to[bend right](-6,-6);
\draw[thick, color = cadmium] (-2,-4) to[bend right](-4,-6);
\draw[thick, color = amber] (-2,-4) to(-2,-6);
\draw[thick, color = forest] (-2,-4) to[bend left](0,-6);
\draw[thick, color = blue] (-2,-4) to[bend left](2,-6);

\draw[thick, color = forest] (-4,-6) -- (-2,-6) node[pos = 0.5, above]{$4$};
\draw[thick, color = blue] (-2,-6) -- (0,-6) node[pos = 0.5, above]{$5$};
\draw[thick, color = amber] (0,-6) -- (2,-6) node[pos = 0.5, above]{$3$};
\draw[thick, color = cadmium] (-2,-6) to[bend right = 50] (2, -6);
\draw[thick, color = cadmium] (0,-7.1) node{$2$};
\draw[thick, color = amber] (-6,-6) -- (-4,-6)node[pos = 0.5, above]{$3$};

\draw[thick, color = red] (-5,-4.5) node{$1$};
\draw[thick, color = cadmium] (-3.5,-4.5) node{$2$};
\draw[thick, color = amber] (-1.8,-4.5) node{$3$};
\draw[thick, color = forest] (-0.5,-4.5) node{$4$};
\draw[thick, color = blue] (1,-4.5) node{$5$};

\filldraw (-2,-4) circle (0.05 cm) node[above]{$v$};
\filldraw (-6,-6) circle (0.05 cm) node[below left]{$v_1$};
\filldraw (-4,-6) circle (0.05 cm) node[below]{$v_2$};
\filldraw (-2,-6) circle (0.05 cm) node[below left]{$v_3$};
\filldraw (0,-6) circle (0.05 cm) node[below]{$v_4$};
\filldraw (2,-6) circle (0.05 cm) node[below right]{$v_5$};

\end{tikzpicture}

\caption{$C$ and $C + v_4v_5$, colored to avoid a rainbow $C_4$-copy} \label{case 5}
\end{center}
\end{figure}
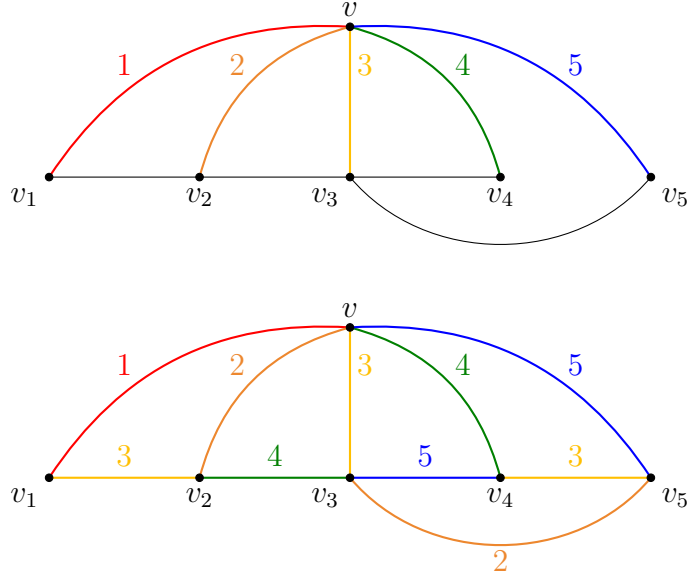

\item $|V(C)| = 5$, and $C$ has diameter $2$

  Observe that this case is impossible, as $C$ is the star $S_4$, say with central vertex $v_1$; we have $|N(v) \cap N(v_1)| = 4$, which is forbidden in a rainbow $C_4$-saturated graph by \cref{observations} Case~\ref{obs:nighbr}.
\hfill $\diamondsuit$
\end{enumerate}
\end{proof}

Now, we combine Claims~\ref{core nbhds} and~\ref{L comps} to bound $|\mathcal{C}_s|$. Our goal is to bound the number of components in $\mathcal{C}_s$ on fewer than $6$ vertices. We have two cases, based upon the outcome of Claim~\ref{core nbhds}. In each, we show that a component $C \in \mathcal{C}_s$ with $|V(C)| < 6$ must either have $|V(C)| = 1$ (which can happen at most once) or correspond to a component $C' \in \mathcal{C}_v$ with $|C'| \leq 2$ induced by some core neighborhood. Since Claim~\ref{L comps} implies that each core neighborhood induces at most two such components, and $|S| \leq 35$, this will imply that at most $71$ components $C \in \mathcal{C}_s$ have $|V(C)| < 6$.

\begin{enumerate}
  \item Claim~\ref{core nbhds} returns $S$ such that $L \subseteq N(S)$. 

    Suppose that $C \in \mathcal{C}_s$ has $2 \leq |V(C)| < 6$. We consider $V(C) \cap L$, which is contained in $N(S)$. If $V(C) \cap L$ intersects two core neighborhoods, say in $w_1 \in N(v) \setminus D$ and $w_2 \in N(u) \setminus D$, then either the component of $w_1$ in $G[N(v) \setminus D]$ has size smaller than $3$, or the component of $w_2$ in $G[N(u) \setminus D]$ does. If $V(C) \cap L$ is contained in one core neighborhood, say $N(v) \setminus D$, then observe that either $C \in \mathcal{C}_v'$ or $V(C)$ induces at least two components of $\mathcal{C}_v$. Indeed, if $V(C)$ induces exactly one component $C' \in \mathcal{C}_v$, then since $C$ is a tree and $C'$ is a connected subgraph of $C$ which contains all the leaves of $C$, we must have $C' = C$. Thus, either $C \in \mathcal{C}_v'$, in which case Claim~\ref{L comps} implies that $C$ is a component of size $2$ contained in a core neighborhood, or $V(C)$ induces at least two components in $N(v) \setminus D$, one of which must contain at most $2$ vertices.

  \item Claim~\ref{core nbhds} returns $S$ such every $C \in \mathcal{C}_s$ with $V(C) \not\subseteq N(S)$ intersects $N(S)$ at least $5$ times.

    In this case, the condition on $S$ guarantees that every $C \in \mathcal{C}_s$ with $V(C) \not\subseteq N(S)$ has $|V(C)| \geq 6$. We thus consider components $C \in \mathcal{C}_s$ with $V(C) \subseteq N(S)$. For such a component $C$, we have $V(C) \cap L \subseteq S$, and similarly to Case 1, if $|V(C)| < 6$, then $C$ must induce a component of size at most $2$ in some core neighborhood.

\end{enumerate}

Thus, in either case, we conclude that $\mathcal{C}_s$ contains at most $71$ components on fewer than $6$ vertices. As a simple lower bound, we thus have 

\[|\mathcal{C}_s| \leq \frac{\left(\sum_{C \in \mathcal{C}_s} |V(C)|\right) - 71}{6} + 71 \leq \frac{|V \setminus D| - 71}{6} + 71 = \frac{|V \setminus D|}{6} + \frac{355}{6}.\]

Thus, using \Cref{eq:edgeDensity}, we have

\[|E(G)| \geq  |E(G[D])| + 2|V \setminus D| - |\mathcal{C}_s| \geq  |E(G[D])| + \frac{11}{6}|V \setminus D| - \frac{355}{6}.\]
Recall that by application of \Cref{dominating sets}, we have $D$ either small or of high average degree.
If $|D| \leq \frac{5}{11} \varepsilon$, then we have 
\[|E(G)| \geq \frac{11}{6} \left( 1 - \frac{5}{11}\varepsilon\right)n - \frac{355}{6} = \frac{11}{6}n - \left( \frac{5\varepsilon}{6}n + \frac{355}{6} \right) > \left(\frac{11}{6} - \varepsilon \right)n.\]
On the other hand, if $G[D]$ has average degree $4$, then 
\[|E(G)| \geq  2|D|  + \frac{11}{6}|V \setminus D| - \frac{355}{6} > \frac{11}{6} n - \frac{355}{6} > \left( \frac{11}{6} - \varepsilon \right)n. \qedhere\]
\end{proof}

\section{Upper bounds for longer cycles}\label{c56 section}

We now exhibit some constructions which provide upper bounds on $\mathrm{sat}^*(n,C_5)$ and $\mathrm{sat}^*(n,C_6)$. Each of these contains a small dominating set of vertices $D$, with the complement of $D$ essentially inducing a disjoint union of cliques. 

\begin{const}\label{c5const}
For $n \geq 8$, let $G(n,C_5)$ be the $n$-vertex graph defined as follows. Designate two vertices $u,v$, which will be universal. Moreover, add a maximum matching in $V(G)\setminus\{u,v\}$.

\end{const}

Observe that $G(n,C_5)$ contains $2n - 3$ edges which are incident to a universal vertex, and $\lfloor \frac{n-2}{2} \rfloor$ edges which are not. Thus, $|E(G_n)| = \lfloor \frac{5n}{2} \rfloor - 4$. 

\fivecycle*

\begin{proof}
It suffices to show that for $n \geq 9$, $G(n,C_5)$ is rainbow $C_5$-saturated. We begin by labelling vertices. Let $u,v$ be the universal vertices of $G(n,C_5)$. We label the $\lfloor \frac{n-2}{2}\rfloor$ edges of $G(n,C_5)$ which are not incident to $u$ or $v$ as $e_i$ ($i \in \{1, \dots \lfloor \frac{n-2}{2}\rfloor\}$), and for each $e_i$, we label the endpoints of $e_i$ as $x_i,y_i$. If $n$ is odd, $G$ contains one vertex which is adjacent only to $u$ and $v$, which we label $z$.  

We first exhibit a proper edge-coloring of $G(n,C_5)$ which is rainbow $C_5$-free. For each pair $x_i,y_i$, we choose $c(ux_i) = c(vy_i) = 2i - 1$ and $c(uy_i) = c(vx_i) = 2i$. If $z$ exists, we choose $c(uz) = n-2$ and $c(vz) = n-1$. We may choose $c(uv)$ and $c(x_iy_i)$ freely; for simplicity, we assign a single color $0$ to $uv$ and each $x_iy_i$. This coloring is depicted in Figure~\ref{c5coloring}.

\begin{figure}[h]
\begin{center}

\begin{tikzpicture}[scale = 1.5]

\draw[thick, color = brilliantrose] (-1,0) -- (1,0);
\draw[color = brilliantrose] (0,0.1) node{\small 0};

\draw[thick, color = brilliantrose] (0,-1) -- (0,-1.5);
\draw[color = brilliantrose] (0.1, -1.25) node{\small 0};

\draw[thick, color = brilliantrose] (0,-2.5) -- (0,-3);
\draw[color = brilliantrose] (0.1, -2.75) node{\small 0};

\draw[thick, color = brilliantrose] (0,-5) -- (0,-5.5);
\draw[color = brilliantrose] (0.1, -5.25) node{\small 0};

\draw[thick, color = red] (-1,0) to[bend right] (0,-1);
\draw[color = red] (-0.5, -0.5) node{\small 1};

\draw[thick, color = cadmium] (0,-1) to[bend right] (1,0);
\draw[color = cadmium] (0.5, -0.5) node{\small 2};

\draw[thick, color = cadmium] (-1,0) to[bend right] (0,-1.5);
\draw[color = cadmium] (-0.5, -1.4) node{\small 2};

\draw[thick, color = red] (0,-1.5) to[bend right] (1,0);
\draw[color = red] (0.5, -1.4) node{\small 1};

\draw[thick, color = amber] (-1,0) to[bend right = 50] (0,-2.5);
\draw[color = amber] (-0.5, -2) node{\small 3};

\draw[thick, color = forest] (0,-2.5) to[bend right = 50] (1,0);
\draw[color = forest] (0.5, -2) node{\small 4};

\draw[thick, color = forest] (-1,0) to[bend right = 50] (0,-3);
\draw[color = forest] (-0.5, -2.9) node{\small 4};

\draw[thick, color = amber] (0,-3) to[bend right = 50] (1,0);
\draw[color = amber] (0.5, -2.9) node{\small 3};

\draw[thick, color = blue] (-1,0) to[bend right = 50] (0,-5);
\draw[color = blue] (-0.7, -4) node{\small $n - 3$};

\draw[thick, color = byzantine] (0,-5) to[bend right = 50] (1,0);
\draw[color = byzantine] (0.7, -4) node{\small $n-2$};

\draw[thick, color = byzantine] (-1,0) to[bend right = 50] (0,-5.5);
\draw[color = byzantine] (-1.1, -5) node{\small $n-2$};

\draw[thick, color = blue] (0,-5.5) to[bend right = 50] (1,0);
\draw[color = blue] (1.1, -5) node{\small $n -3$};

\filldraw (-1,0) circle (0.05 cm);
\filldraw (1,0) circle (0.05 cm);
\filldraw (0,-1) circle (0.05 cm);
\filldraw (0,-1.5) circle (0.05 cm);
\filldraw (0,-2.5) circle (0.05 cm);
\filldraw (0,-3) circle (0.05 cm);
\draw (0,-4) node{$\vdots$};
\filldraw (0,-5) circle (0.05 cm);
\filldraw (0,-5.5) circle (0.05 cm);

\end{tikzpicture}

\caption{An edge-coloring of $G(n,C_5)$ containing no rainbow $C_5$-copy}\label{c5coloring}
\end{center}

\end{figure}
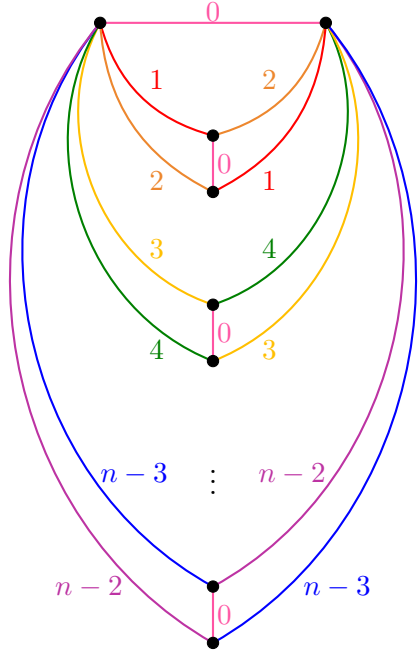

To see that the described edge-coloring of $G(n,C_5)$ is rainbow $C_5$-free, observe first that every $C_5$-copy in $G(n,C_5)$ contains a copy of $P_4$ with endpoints $u$ and $v$. Observe that in the described edge-coloring, no $P_4$-copy between $u$ and $v$ is rainbow.

We now show that $G(n,C_5)$ is rainbow $C_5$-saturated. Let $G'$ be a supergraph of $G(n,C_5)$ (possibly with $G' = G(n,C_5)$). Observe that if a proper edge-coloring of $G'$ contains a rainbow $P_4$-copy with endpoints $u,v$, then under this edge-coloring, $G'$ contains a rainbow $C_5$-copy. Indeed, suppose such a rainbow $P_4$-copy exists, say $u x_i y_i v$, with edge colors $c(ux) = 1, c(xy) = 2, c(y v) = 3$. Now, since $n \geq 9$, $S \coloneqq |(N(u)\cap N(v)) \setminus \{x,y\}| \geq 5$. For any vertex $w \in S$, either $u x y v w u$ is a rainbow $C_5$-copy or else one of $c(uw), c(vw)$ is in $\{1,2,3\}$. Since the edge-coloring under consideration is proper, we cannot have $c(uw) = 1$ or $c(vw) = 3$, and so one of four outcomes must hold: either $c(uw) \in \{2,3\}$ or $c(vw) \in \{1,2\}$. Each of these four outcomes holds for at most one vertex $w \in S$. Thus, as $|S| \geq 5$, under the given edge-coloring, $G'$ contains at least one rainbow $C_5$-copy. 

By this observation, it suffices to show that if any edge $e$ is added to $G(n,C_5)$, then every proper edge-coloring of $G(n,C_5) + e$ contains a rainbow $P_4$-copy with endpoints $u,v$. We now establish this fact. Suppose for a contradiction that there exists a proper edge-coloring of $G(n,C_5) + e$ so that no copy of $P_4$ with endpoints $u,v$ is rainbow. For every $i \in \{1, \dots \lfloor \frac{n-2}{2}\rfloor\}$, we have that $u x_i y_i v$ and $u y_i x_i v$ are $P_4$-copies with endpoints $u,v$, and so to ensure that neither is rainbow, we must have $c(ux_i) = c(vy_i)$ and $c(uy_i) = c(vx_i)$. Now, consider the endpoints of $e$. Without loss of generality, $e$ is incident to $x_1$. We shall (re)label the other endpoint of $e$ as $w$; note that, since $e \not\in E(G(n,C_5))$, we must have $w \not\in \{u,v, y_1\}$. Now, consider the path $u x_1 w v$. We know that $c(u x_1) \neq c(wv)$, since $c(y_1 v) = c(ux_1)$ and $y_1 \neq w$. Thus, if $u x_1 w v$ is not a rainbow $P_4$-copy, then either $c(ux_1) = c(x_1 w)$ or $c(x_1w) = c(wv)$. However, either of these equalities implies that the edge-coloring under consideration is not proper, a contradiction.
\end{proof}

Next, we present a construction which gives an upper bound on $\mathrm{sat}^*(n,C_6)$.
An obvious analog to Construction~\ref{c5const} would be to take a small set of universal vertices and tile their common neighborhood with copies of $K_3$. However, for large enough $n$, such a graph cannot be edge-colored to avoid a rainbow $C_6$-copy. Instead of a pair of universal vertices, our construction features a small dominating set. In order to achieve saturation, we also must introduce a more complicated structure among a small subset of the vertices which are not in this dominating set.

\begin{const}\label{c6 const}
For $n \equiv 2 \mod 3$ and $n \geq 14$, let $G(n,C_6)$ be the $n$-vertex graph defined as follows. We designate eight vertices $\{v_1, v_2, \dots, v_8\}$, which we shall term the \textit{core vertices}. The core vertices induce a copy of $K_8$ with the matching $\{v_1v_3, v_2v_8, v_4v_6, v_5v_7\}$ removed. We partition $V(G) \setminus \{v_1, \dots, v_8\}$ into $\frac{n-8}{3}$ sets of size $3$. We label these sets as $T_i = \{x_i,y_i,z_i\}$. For each $i \in \{1, \dots, \frac{n-8}{3}\}$, we add all edges between vertices in $T_i$, as well as the edges $v_1x_i, v_1y_i, v_2 z_i, v_3 y_i$.

\end{const}

When $n \equiv i \mod 3$ with $i \neq 2$ we let $G(n,C_6)$ be the graph obtained from Construction~\ref{c6 const} by forming $G( n - 2 + i, C_6)$ and then adding $2 -i$ additional vertices, adjacent to $v_1, v_3$ ,and (if $2 - i = 2$) to one another. Note that in any congruence case, $E(G(n,C_6)) = \frac{7}{3}n + O(1)$. 

We start by describing an edge-coloring under which $G(n,C_6)$ is rainbow-$C_6$-free. Color the core using the edge-coloring exhibited in Figure~\ref{core coloring}(a), which can be viewed as a decomposition of the core into six perfect matchings. For $i \in \{1, \dots, \frac{n-8}{3}\}$, we choose $c(v_1 x_i) = 3 + 3i$, $c(v_1 y_i) = 4 + 3i$, $c(v_2 z_i) = 3 + 3i$, $c(v_3 y_i) = 3 + 3i$, and $c(y_i z_i) = 0$, $c(x_i z_i) = 4 + 3i$, $c(x_i y_i) = 5 + 3i$. Finally, if $n \not\equiv 2\mod 3$, we must consider the additional vertices adjacent to $v_1$ and $v_3$. If there is one such additional vertex, then edges incident to this vertex may be colored in any way which maintains a proper edge-coloring. It there are two, say $s,t$, we choose $c(v_1 s) = c(v_3 t)$ and $c(v_1 t) = c(v_3 s)$, which ensures that neither $v_1 s t v_3$ nor $v_1 t s v_3$ is a rainbow $P_4$-copy. We depict representative sets of non-core vertices $T_i$, with the described edge-coloring, in Figure~\ref{core coloring}(b).

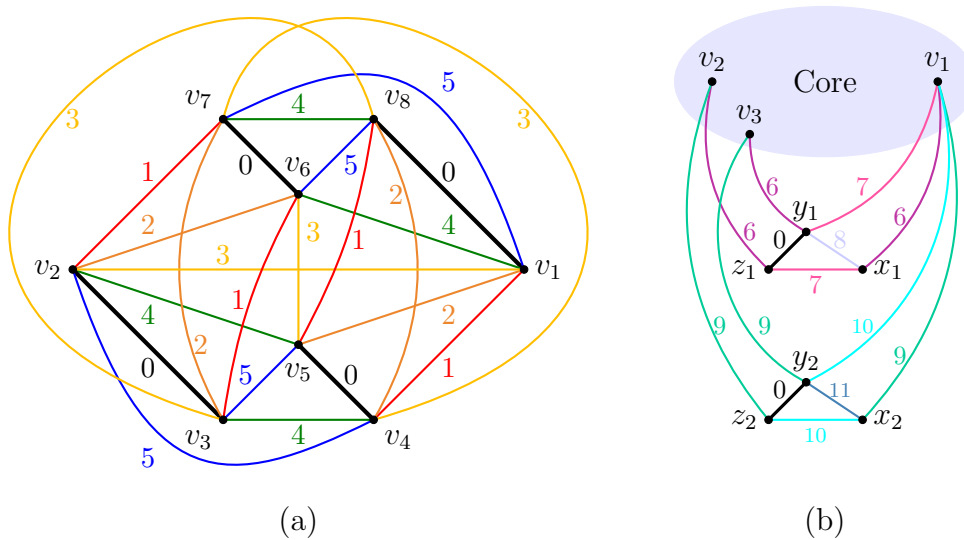
\begin{figure}[h]
\begin{center}

\begin{tikzpicture}

\draw 
(0,-2) node[below]{(a)}
(7,-2) node[below]{(b)}
;
\clip(-4,-2) rectangle (9.2,5);

\draw[thick, color = red] (-3,1) -- (-1,3);
\draw[color = red] (-2,2.3) node{\small 1}; 

\draw[thick, color = cadmium] (-3,1) -- (0,2);
\draw[color = cadmium] (-2, 1.6) node{\small 2}; 

\draw[thick, color = forest] (-3,1) -- (0,0);
\draw[color = forest] (-2, 0.4) node{\small 4};

\draw[ultra thick] (-3,1) -- (-1,-1);
\draw (-2,-0.3) node{\small 0};

\draw[ultra thick] (3,1) -- (1,3);
\draw (2,2.3) node{\small 0};

\draw[thick, color = forest] (3,1) -- (0,2);
\draw[color = forest] (2,1.6) node{\small 4};

\draw[thick, color = cadmium] (3,1) -- (0,0);
\draw[color = cadmium] (2,0.4) node{\small 2};

\draw[thick, color = red] (3,1) -- (1,-1);
\draw[color = red] (2,-0.3) node{\small 1};

\draw[thick, color = blue] (0,0) -- (-1,-1);
\draw[color = blue] (-0.7, -0.4) node{\small 5};

\draw[ultra thick] (0,0) -- (1,-1);
\draw (0.7, -0.4) node{\small 0};

\draw[ultra thick] (0,2) -- (-1,3);
\draw (-0.7, 2.4) node{\small 0};

\draw[thick, color = blue] (0,2) -- (1,3);
\draw[color = blue] (0.7, 2.4) node{\small 5};

\draw[thick, color = blue] (-3,1) .. controls (-2,-2) and (-1,-2) .. (1,-1);
\draw[color = blue] (-2,-1.5) node{\small 5};

\draw[thick, color = blue] (3,1) .. controls (2,4) and (1,4) .. (-1,3);
\draw[color = blue] (2,3.5) node{\small 5};

\draw[thick, color = cadmium] (-1,3) to[bend right] (-1,-1);
\draw[color = cadmium] (-1.3, 0) node{\small 2};

\draw[thick, color = cadmium] (1,3) to[bend left] (1,-1);
\draw[color = cadmium] (1.3,2) node{\small 2};

\draw[thick, color = amber] (-3,1)-- (3,1);
\draw[color = amber] (-1,1.2) node{\small 3};

\draw[thick, color = forest] (-1,3)-- (1,3); 
\draw[color = forest] (0,3.2) node{\small 4};

\draw[thick, color = amber] (0,2)-- (0,0);
\draw[color = amber] (0.2, 1.5) node{\small 3};


\draw[thick, color = amber] (1,3).. controls (0,7) and (-8, 1).. (-1,-1);
\draw[color = amber] (-3,3) node{\small 3};

\draw[thick, color = amber] (-1,3).. controls (0,7) and (8, 1).. (1,-1);
\draw[color = amber] (3,3) node{\small 3};

\draw[thick, color = red] (1,3) to[bend left = 10] (0,0);
\draw[color = red] (0.8, 1.4) node{\small 1};

\draw[thick, color = red] (-1,-1) to[bend left = 10] (0,2);
\draw[color = red] (-0.8, 0.6) node{\small 1};

\draw[thick, color = forest] (-1,-1)-- (1,-1);
\draw[color = forest] (0,-1.2) node{\small 4};

\filldraw (-1,3) circle (0.05 cm) node[above left]{$v_7$};
\filldraw (1,3) circle (0.05 cm) node[above right]{$v_8$};
\filldraw (0,2) circle (0.05 cm);
\draw (0,2.1) node[above]{$v_6$};
\filldraw (-3,1) circle (0.05 cm) node[left]{$v_2$};
\filldraw (3,1) circle (0.05 cm) node[right]{$v_1$};
\filldraw (0,0) circle (0.05 cm);
\draw (0,-0.1) node[below]{$v_5$};
\filldraw (1,-1) circle (0.05 cm) node[below right]{$v_4$};
\filldraw (-1,-1) circle (0.05 cm) node[below left]{$v_3$};


\begin{scope}[xshift=7cm,yshift=3.5cm]
\filldraw[color = blue!10!white, fill = blue!10!white] (0,0) ellipse (2  and 1 );
\draw (0,0) node{Core};

\draw[thick, color = byzantine] (1.5,0) to[bend left = 30] (0.5,-2.5);
\draw[color = byzantine] (1,-1.8) node{\footnotesize 6};

\draw[thick, color = brilliantrose] (1.5,0) to[bend left = 30] (-0.25,-2);
\draw[color = brilliantrose] (0.5,-1.4) node{\footnotesize 7};

\draw[thick, color = byzantine] (-1,-0.7) to[bend right = 30](-0.25,-2);
\draw[color = byzantine] (-0.7,-1.4) node{\footnotesize 6};

\draw[thick, color = byzantine] (-1.5,0) to[bend right = 30] (-0.75,-2.5);
\draw[color = byzantine] (-1,-2) node{\footnotesize 6};

\draw[thick, color = periwinkle] (0.5, -2.5) -- (-0.25, -2);
\draw[color = periwinkle] (0.2,-2.1) node{\footnotesize 8};

\draw[thick, color = brilliantrose] (0.5, -2.5) -- (-0.75, -2.5);
\draw[color = brilliantrose] (-0.125,-2.7) node{\footnotesize 7};

\draw[very thick] (-0.75, -2.5) -- (-0.25, -2);
\draw (-0.6,-2.1) node{\footnotesize 0};

\draw[thick, color = caribbeangreen] (1.5,0) to[bend left = 30] (0.5,-4.5);
\draw[color = caribbeangreen] (-0.8,-3.25) node{\footnotesize 9};

\draw[thick, color = electriccyan] (1.5,0) to[bend left = 40] (-0.25,-4);
\draw[color = electriccyan] (0.5,-3.25) node{\scriptsize 10};

\draw[thick, color = caribbeangreen] (-1,-0.7) to[bend right = 50](-0.25,-4);
\draw[color = caribbeangreen] (-1.4,-3.25) node{\footnotesize 9};

\draw[thick, color = caribbeangreen] (-1.5,0) to[bend right = 30] (-0.75,-4.5);
\draw[color = caribbeangreen] (1,-3.65) node{\footnotesize 9};

\draw[thick, color = steelblue] (0.5, -4.5) -- (-0.25, -4);
\draw[color = steelblue] (0.2,-4.1) node{\scriptsize 11};

\draw[thick, color = electriccyan] (0.5, -4.5) -- (-0.75, -4.5);
\draw[color = electriccyan] (-0.125,-4.7) node{\scriptsize 10};

\draw[very thick] (-0.75, -4.5) -- (-0.25, -4);
\draw (-0.6,-4.1) node{\footnotesize 0};

\filldraw (-1.5, 0) circle (0.05 cm);
\draw (-1.5, 0) node[above]{$v_2$};
\filldraw (1.5, 0) circle (0.05 cm);
\draw (1.5, 0) node[above]{$v_1$};
\filldraw (-1, -0.7) circle (0.05 cm);
\draw (-1, -0.7) node[above]{$v_3$};

\filldraw (0.5,-2.5) circle(0.05 cm);
\draw (0.5, -2.5) node[right]{$x_1$};

\filldraw (-0.25,-2) circle(0.05 cm);
\draw (-0.25, -2) node[above]{$y_1$};

\filldraw (-0.75,-2.5) circle(0.05 cm);
\draw (-0.75, -2.5) node[left]{$z_1$};

\filldraw (0.5,-4.5) circle(0.05 cm);
\draw (0.5, -4.5) node[right]{$x_2$};

\filldraw (-0.25,-4) circle(0.05 cm);
\draw (-0.25, -4) node[above]{$y_2$};

\filldraw (-0.75,-4.5) circle(0.05 cm);
\draw (-0.75, -4.5) node[left]{$z_2$};
\end{scope}
\end{tikzpicture}

\caption{(a) An admissible edge-coloring of the core.
(b) Non-core triangles and their connections to the core.}\label{core coloring}
\end{center}
\end{figure}


Recall, our goal is to demonstrate the following bound.

\sixcycle*

It would suffice to show that $G(n,C_6)$ is properly rainbow $C_6$-saturated. It is possible to check this by hand; however, the verification is somewhat lengthy and requires a fair amount of case analysis. Thus, we elect to verify the desired properties using computer assistance. 

We used SageMath~\cite{sagemath}
to create SAT formulas that were solved by PicoSAT~\cite{Biere2008} and Glucose~\cite{Glucose}.
A Jupyter Notebook with our calculations is available on arXiv.
We use the straightforward approach where each edge $e\in E$ and possible color $c\in C$ gives a boolean variable $x_e^c$ indicating if the edge $e$ is colored by the color $c$. 
The SAT formula for checking the existence of a proper $C$-coloring avoiding a rainbow $C_6$-copy is the following. 
\[
\underbrace{\bigwedge_{e\in E}\left( \bigvee_{c \in C}  x_e^c \right)}_{\text{every edge has a color}} 
\wedge
\underbrace{
\bigwedge_{\substack{e_1 \cap e_2 \neq \emptyset \\e_1,e_2 \in E }} \bigwedge_{c\in C} ( \neg x_{e_1}^c  \vee  \neg x_{e_2}^c )}_{\text{adjacent edges are colored differently}}
\wedge
\underbrace{
\bigwedge_{ \substack{\text{6-cycles} \\e_1,\ldots,e_6 \in E} }
\bigwedge_{\substack{\text{distinct}\\ c_1,\ldots,c_6\in C}}
\left(
\bigvee_{i=1}^{6}  \neg x_{e_i}^{c_i} 
\right)}_{\text{No rainbow 6-cycles}}.
\]
It is sufficient to do the calculation on
a graph $H$ consisting of the core and two triangles $T_1$ and $T_2$.
We show that $H$ can be colored by 10, 11 or 12 colors while avoiding a rainbow $C_6$-copy, but it cannot be colored using 13 or more colors. In order to show 13 or more colors are not possible, we need to add two more types of clauses to the formula, forcing that all colors are used.
\[
\underbrace{\bigwedge_{c\in C}\left( \bigvee_{e \in E}  x_e^c \right)}_{\text{every color is used}} 
\wedge
\underbrace{\bigwedge_{e\in E}
\bigwedge_{c_1 \neq c_2 \in C}
\left( \neg x_e^{c_1}  \vee  \neg x_e^{c_1} \right)}_{\text{every edge has at most one color}} .
\]
The resulting formula needs a fixed number of colors. 
We first show that the core can be colored by 6 or 7 colors but not by any larger number. 
It suffices to consider number of colors up to 15 via the following arguments. Fix a proper-edge coloring without a rainbow 6-cycle.
Iterating over 6-cycles $H$ remove one of the edges whose color is not unique in $H$. 
This eventually results in a 6-cycle-free graph. 
The largest 6-cycle-free subgraph of the core has 15 edges which can be showed by an integer program. 
Adding one triangle to the core can increase the number of colors by at most 7. By testing up to 14 colors we get that core and one triangle can use at most 9 colors. Hence core and two triangles can use at most 12 colors.
Finally, we can upper bound the number of colors by 13 when showing that adding any missing edge to $H$ results in a rainbow $C_6$ in any proper coloring.

To deal with $n \equiv i \mod 3$ with $i \neq 2$, we use a graph $F$ consisting of a core, one triangle $T_1$ and an extra vertex of degree 2 adjacent to $v_1$ and $v_3$. 
It is not possible to add any edge to $F$ where a proper edge-coloring does not create a rainbow 6-cycle.

The entire calculation can be done in one day.
\section{Conclusion}
While we determined the asymptotics for 
$\mathrm{sat}^*(n,C_4)$
and 
constructed upper bounds for 
$\mathrm{sat}^*(n,C_5)$ and 
$\mathrm{sat}^*(n,C_6)$,
finding exact bounds is still open.

\begin{prob}
    Determine $\mathrm{sat}^*(n,C_k)$ for all $k$.
\end{prob}

When using SAT solver in the previous section, we encountered the following problem. 
A positive answer would simplify the computation.
\begin{prob}
    Let $G$ and $H$ be finite graphs.
    Let $K$ be the set of all $k$ where exists a proper $k$-edge-coloring of $G$ using all $k$ colors and avoiding a rainbow copy of $H$. Is $K$ an interval?
\end{prob}

\section*{Acknowledgement}
This research was supported by NSF DSM-2152490 grant.
Tomáš Masařík was supported by Polish National Science Centre SONATA-17 grant number 2021/43/D/ST6/03312.

\bibliographystyle{abbrvurl}
\bibliography{references}

\end{document}